\documentclass[11pt]{amsart}

\usepackage{xypic,amsmath,amssymb,hyperref}

\newtheorem{proposition}{Proposition}[section]
\newtheorem{lemma}[proposition]{Lemma}
\newtheorem{corollary}[proposition]{Corollary}
\newtheorem{theorem}[proposition]{Theorem}

\theoremstyle{definition}
\newtheorem{definition}[proposition]{Definition}
\newtheorem{example}[proposition]{Example}

\theoremstyle{remark}
\newtheorem{remark}[proposition]{Remark}
\newtheorem{remarks}[proposition]{Remarks}

\newcommand{\thlabel}[1]{\label{th:#1}}
\newcommand{\thref}[1]{Theorem~\ref{th:#1}}
\newcommand{\selabel}[1]{\label{se:#1}}
\newcommand{\seref}[1]{Section~\ref{se:#1}}
\newcommand{\lelabel}[1]{\label{le:#1}}

\newcommand{\prlabel}[1]{\label{pr:#1}}
\newcommand{\prref}[1]{Proposition~\ref{pr:#1}}
\newcommand{\colabel}[1]{\label{co:#1}}
\newcommand{\coref}[1]{Corollary~\ref{co:#1}}

\newcommand{\exlabel}[1]{\label{ex:#1}}
\newcommand{\exref}[1]{Example~\ref{ex:#1}}
\newcommand{\delabel}[1]{\label{de:#1}}
\newcommand{\deref}[1]{Definition~\ref{de:#1}}
\newcommand{\eqlabel}[1]{\label{eq:#1}}
\newcommand{\equref}[1]{(\ref{eq:#1})}

\def\equal#1{\smash{\mathop{=}\limits^{#1}}}

\newcommand{\can}{{\rm can}}

\newcommand{\Hom}{{\rm Hom}}

\newcommand{\Ker}{{\rm Ker}}
\newcommand{\Coker}{{\rm Coker}\,}

\def\Sets{\underline{\rm Sets}}
\def\Cog{\underline{\Cc}}
\def\Alg{\underline{\Aa}}
\def\Cat{\dul{\rm Cat}}
\def\gr{\dul{\rm gr}}
\def\lan{\langle}
\def\ran{\rangle}
\def\ot{\otimes}

\newcommand{\ev}{{\rm ev}}

\def\sq{\square}

\def\rightact{\hbox{$\leftharpoonup$}}

\newcommand{\Aa}{\mathcal{A}}

\newcommand{\Cc}{\mathcal{C}}
\newcommand{\Dd}{\mathcal{D}}

\newcommand{\Mm}{\mathcal{M}}

\newcommand{\Vv}{\mathcal{V}}
\newcommand{\Ww}{\mathcal{W}}

\def\*C{{}^*\hspace*{-1pt}{\Cc}}

\def\text#1{{\rm {\rm #1}}}

\def\ol{\overline}
\def\ul{\underline}
\def\dul#1{\underline{\underline{#1}}}

\allowdisplaybreaks[4]

\begin{document}
\title[Hopf categories]{Hopf categories}
\author[E. Batista]{E. Batista}
\address{Departamento de Matem\'atica, Universidade Federal de Santa Catarina, Brazil}
\email{ebatista@mtm.ufsc.br}
\author{S. Caenepeel}
\address{Faculty of Engineering,
Vrije Universiteit Brussel, Pleinlaan 2, B-1050 Brussels, Belgium}
\email{scaenepe@vub.ac.be}
\author[J. Vercruysse]{J. Vercruysse}
\address{D\'epartement de Math\'ematique, Universit\'e Libre de Bruxelles, Belgium}
\email{jvercruy@ulb.ac.be}



\thanks{\rm 
The second author was supported by the research project G.0117.10  
``Equivariant Brauer groups and Galois deformations'' from
FWO-Vlaanderen. The third author wants to thank the FWB (F\'ed\'eration Wallonie-Bruxelles)
for the support on the ARC-project ``Hopf algebras and symmetries of non-commutative spacesÕ'.}

\begin{abstract}
We introduce Hopf categories enriched over braided monoidal categories. The notion is
linked to several recently developed notions in Hopf algebra theory, such as Hopf group
(co)algebras, weak Hopf algebras and duoidal categories. We generalize the fundamental
theorem for Hopf modules and some of its applications to Hopf categories.
\end{abstract}
\maketitle

\section*{Introduction}\selabel{0}
The starting point of this paper is enriched category theory. Given a (strict) monoidal
category $\Vv$, we can consider the notion of $\Vv$-category. For example, if
$\Vv$ is the category of sets, then a $\Vv$-category is an ordinary category.
If $\Vv$ is the category of vector spaces, then a $\Vv$-category is a linear category.
A $\Vv$-category with one object is an algebra (or monoid) in $\Vv$.\\
Now consider a braided monoidal category. The category $\Cog(\Vv)$ of coalgebras in
$\Vv$ is a monoidal category, so we can consider $\Cog(\Vv)$-categories. 
A Hopf $\Vv$-category is a $\Cog(\Vv)$-category with an antipode. These definitions
are designed in such a way that $\Cog(\Vv)$-categories, resp. Hopf $\Vv$-categories,
with one object correspond to bialgebras, resp. Hopf algebras in $\Vv$.
In the world of sets, the notion is not of great interest, since $\Cog(\Sets)=\Sets$:
it is well-known that every set has a unique structure of a coalgebra in $\Sets$.
Hopf categories are groupoids, that is, categories
in which every morphism is invertible. In fact, $\Cog(\Vv)$-categories only
come to life when we pass to the $k$-linear world!\\
Hopf categories are related to several recent generalizations of Hopf algebras and
monoidal categories. For example, Hopf group algebras and Hopf group coalgebras
give rise to examples of Hopf categories, respectively over the
category of vector spaces and its dual category, see \seref{4}. In \seref{6} we will 
show that $k$-linear Hopf categories with a set of objects
are Hopf monoids in the sense of \cite{BL}
(in particular bimonoids in the sense of \cite{Aguiar,BCZ}) in a suitable duoidal category. This also indicates the relation with other generalized Hopf-like structures, such as Hopf monads \cite{BLV}.\\
Hopf categories with a finite number of objects can be used to construct examples of weak Hopf algebras,
see \seref{5}. As we have mentioned above, groupoids are Hopf categories over sets. Applying
the linearization functor, we obtain a Hopf category over the category of vector spaces,
Putting this into packed form, we obtain a weak Hopf algebra, which turns out to be the groupoid
algebra, the basic example of a weak Hopf algebra.\\
This brings us to duality. The second author made attempts to construct a satisfactory
duality theory for group algebras, based on the philosophy developed in \cite{CL}. For
Hopf categories, duality works. The dual of a (finite) Hopf $\Mm_k$-category (also termed a
$k$-linear Hopf category) is a Hopf $\Mm_k^{\rm op}$-category, see Theorems \ref{th:3b.3} and \ref{th:3b.4}.
We also have a categorical version of the well-known property that $C$-comodules correspond
to $C^*$-modules, in the case where $C$ is a finitely generated projective coalgebra, see
\prref{3b.5}.\\
It also turns out that some well-known results about Hopf algebras can be generalized
to Hopf categories. We mention a few first results. We have a categorical version of the
important fact that the representation category of a bialgebra carries a monoidal structure,
see \seref{3}. The fundamental theorem extends to Hopf categories, see \seref{9}.\\
It is well-known that Morita contexts can be viewed as $k$-linear categories with two
objects. This is the starting point of \seref{6b}, where the relationship between
Hopf categories, $H$-Galois objects and Morita theory is investigated.
It is possible to develop descent and Galois theory for Hopf categories, this is the
topic of a forthcoming paper. Hopf categories are also related to partial
actions of groups and Hopf algebras (see \cite{ABV,D,DE,KL}), this will be investigated in \cite{BCV2}.

\section{Preliminary results on enriched category theory}\selabel{1}
Let $(\Vv,\ot,k)$ be a monoidal category. We will assume that $\Vv$
is strict. Our results extend easily to arbitrary monoidal categories, in view of the classical
result that every monoidal category is equivalent to a strict one, see for example \cite{K}.
For a class $X$, we construct a new monoidal
category $\Vv(X)$. An object is a family of objects $M$ in $\Vv$ indexed by $X\times X$:
$$M=(M_{x,y})_{x,y\in X}.$$
A morphism $\varphi:\ M\to N$ consists of a family of morphisms $\varphi_{x,y}:\ M_{x,y}\to N_{x,y}$
in $\Vv$, indexed by $X\times X$. The tensor product $M\bullet N$ is defined by the formula
$$(M\bullet N)_{x,y}=M_{x,y} \ot N_{x,y},$$
and the unit object is $J$, with $J_{x,y}=k$, for all $x,y\in X$. To make our notation more transparent,
we will write $J_{x,y}=ke_{x,y}$, where $e_{x,y}$ can be viewed as an elementary matrix.

We have a functor $(-)^{\rm op}:\ \Vv(X)\to \Vv(X)$. The opposite $V^{\rm op}$ of an object $V\in \Vv(X)$
is given by $V^{\rm op}_{y,x}=V_{x,y}$, for all $x,y\in X$, and the opposite  $\varphi^{\rm op}$ of a
morphism $\varphi$ is given by $\varphi^{\rm op}_{y,x}=\varphi_{x,y}$.

From \cite[Sec. 6.2]{B}, we recall the notion of a
$\Vv$-category. A $\Vv$-category $A$ consists of a class $|A|=X$, and an object $A\in \Vv(X)$
together with two classes of morphisms in $\Vv$, namely,
\begin{enumerate}
\item the multiplication morphisms $m=m_{x,y,z}:\ A_{x,y}\ot A_{y,z}\to A_{x,z}$, defined for each $x,y,z\in X$;
\item unit morphisms $\eta_x:\ J_{x,x}=ke_{x,x}\to A_{x,x}$, defined for each $x\in X$,
\end{enumerate}
such that the following associativity and unit conditions are satisfied:
\begin{eqnarray}
&&m_{x,y,t}\circ (A_{x,y}\ot m_{y,z,t})=m_{x,z,t}\circ (m_{x,y,z}\ot A_{z,t})=m^2_{x,y,z,t};\eqlabel{1.1}\\
&&m_{x,x,y}\circ (\eta_x\ot A_{x,y})=A_{x,y}=m_{x,y,y}\circ ( A_{x,y}\ot \eta_y)\eqlabel{1.2}.
\end{eqnarray}
Observe that $J$ is a $\Vv$-category; the multiplication maps $ke_{x,y}\ot ke_{y,z}\to ke_{x,z}$
and the unit maps $ke_{x,x}\to ke_{x,x}$ are all the identity maps.

If $(\Vv,\ot,k)=(\Sets,\times,\{*\})$, then a $\Vv$-category is an ordinary category. Indeed, for
a $\Sets$-category $A$ with underlying class $X$, set $\Hom_A(x,y)=A_{y,x}$. For
$a\in \Hom_A(x,y)=A_{y,x}$ and $b\in \Hom_A(y,z)=A_{z,y}$, we define the composition
$b\circ a=m_{z,y,x}(b, a)$. The unit morphism in $\Hom_A(x,x)=A_{x,x}$ is
$\eta_x(*)$.

If $(\Vv,\ot,k)=(\Mm_k,\ot,k)$, the category of modules over a commutative ring $k$, then
a $\Vv$-category is also called a $k$-linear category.

If $(\Vv,\ot,k,c)$ is a braided monoidal category, then the tensor product $A\bullet B$ in $\Vv(X)$ of two
$\Vv$-categories $A$ and $B$ is again a $\Vv$-category: the multiplication morphisms are
the compositions
\begin{eqnarray*}
&&\hspace*{-6mm}m^{A\bullet B}_{x,y,z}=
(m_{x,y,z}\ot m_{x,y,z})\circ (A_{x,y}\ot c_{B_{x,y},A_{y,z}}\ot B_{y,z}):\\
&&A_{x,y}\ot B_{x,y}\ot A_{y,z}\ot  B_{y,z}\to A_{x,y}\ot A_{y,z}\ot B_{x,y} \ot  B_{y,z}
\to A_{x,z}\ot B_{x,z}.
\end{eqnarray*}

$\Vv$-categories can be organized into a 2-category ${}_{\Vv}\dul{\rm Cat}$.\\
Let $A$ and $B$ be $\Vv$-categories, with underlying classes $|A|=X$ and $|B|=Y$. 
A  $\Vv$-functor $f:\ A\to B$
consists of the following data: for each $x\in X$, we have $f(x)\in Y$, and we have morphisms
$$f_{x,y}:\ A_{x,y}\to B_{f(x),f(y)}$$
in $\Vv$ such that the following diagrams commute, for all $x,y,z\in X$:
\begin{equation}\eqlabel{Vfunctor}
\xymatrix{
A_{x,y}\ot A_{y,z}\ar[d]_{f_{x,y}\ot f_{y,z}}\ar[rr]^(.6){m_{x,y,z}}&&A_{x,z}\ar[d]^{f_{x,z}}\\
B_{f(x),f(y)}\ot B_{f(y),f(z)}\ar[rr]^(.6){m_{f(x),f(y),f(z)}}&&B_{f(x),f(z)}}\hspace*{5mm}
\xymatrix{ke_{x,x}\ar[dr]_{\eta_{f(x)}}\ar[r]^{\eta_x}&A_{x,x}\ar[d]^{f_{x,x}}\\
&B_{f(x),f(x)}}
\end{equation}
Let $f,g:\ A\to B$ be $\Vv$-functors. A $\Vv$-natural transformation $\alpha:\ f\Rightarrow g$
consists of a class of morphisms $\alpha_x:\ k\to B_{g(x),f(x)}$ in $\Vv$ such that the diagrams
$$\xymatrix{
A_{x,y}\ar[rr]^(.4){g_{x,y}\ot \alpha_y}\ar[d]_{\alpha_x\ot f_{x,y}}&&
B_{g(x),g(y)}\ot B_{g(y),f(y)}\ar[d]^{m_{g(x),g(y),f(y)}}\\
B_{g(x),f(x)}\ot B_{f(x),f(y)}\ar[rr]^(.6){m_{g(x),f(x),f(y)}}&&B_{g(x),f(y)}}$$
commute, for all $x,y\in X$. We have a 2-category ${}_{\Vv}\dul{\rm Cat}$ with $\Vv$-categories,
$\Vv$-functors and $\Vv$-natural transformation as 0-cells, 1-cells and 2-cells. Let us
describe the composition of 1-cells and 2-cells. Given 1-cells $f,~f':\ A\to B$ and $g,~g':\ B\to C$,
$g\circ f:\ A\to C$ is given by the formulas
$$(g\circ f)_{x,y}= g_{f(x),f(y)}\circ f_{x,y}:\ A_{x,y}\to C_{(g\circ f)(x),(g\circ f)(y)}.$$
Now consider 2-cells $\alpha:\ f\Rightarrow f'$ and $\beta:\ g\Rightarrow g'$.
$\alpha*\beta:\ g\circ f\Rightarrow g'\circ f'$ is defined as follows:
\begin{eqnarray*}
(\alpha*\beta)_{x} & = & m_{g'(f'(x)),g'(f(x)),g(f(x))}\circ ((g'_{f'(x),f(x)}\circ\alpha_x)\ot
\beta_{f(x)}) \\
& = & m_{g'(f'(x)),g(f'(x)),g(f(x))}\circ (\beta_{f'(x)} \otimes (g_{f'(x),f(x)} \circ \alpha_x ))
\end{eqnarray*}
Now let $f,g,h:\ A\to B$ be 1-cells, and let $\alpha:\ f\Rightarrow g$, $\beta:\ g\Rightarrow h$
be 2-cells. We define the vertical decomposition $\beta\circ \alpha:\ f\Rightarrow h$ by the rule
$$(\beta\circ \alpha)_x=m_{h(x),g(x),f(x)}\circ (\beta_x\ot \alpha_x).$$

Now fix a class $X$. A $\Vv$-category with underlying class $X$ is called a $\Vv$-$X$-category.
A $\Vv$-functor $f:\ A\to B$ between two $\Vv$-$X$-categories $A$ and $B$ is called a
$\Vv$-$X$-functor if $f(x)=x$ for all $x\in X$, that is, $f$ is the identity on objects.
${}_\Vv\Cat(X)$ is the 2-subcategory of ${}_\Vv\Cat$ with $\Vv$-$X$-categories as 0-cells,
$\Vv$-$X$-functors as 1-cells and $\Vv$-natural transformations as 2-cells.

If $X$ is a singleton, then the 0-cells and 1-cells of ${}_\Vv\Cat(X)$ are $\Vv$-algebras and
$\Vv$-algebra morphisms. A 2-cell $\alpha:\ f\Rightarrow g$ between two algebra morphisms
$f,g:\ A\to B$ is a morphism $\alpha:\ k\to B$ such that $m\circ (g\ot\alpha)=m\circ (\alpha\ot f)$.

Consider the particular situation where $\Vv=\Mm_k$. Then morphisms $\alpha_x:\
k\to B_{x,x}$ correspond to elements $\alpha_x\in B_{x,x}$, and a 2-cell $\alpha:\ f\Rightarrow
g$ between
two $k$-linear $X$-functors consists of elements $\alpha_x\in B_{x,x}$ such that
\begin{equation}\eqlabel{1.3}
g_{x,y}(a)\alpha_y =\alpha_x f_{x,y}(a),
\end{equation}
for all $a\in A_{x,y}$ and $x,y\in X$.

Let $(\Vv,\ot, k)$ and $(\Ww,\sq,l)$ be two strict monoidal categories. Recall that a monoidal
functor $\Vv\to\Ww$ is a triple $(F,\varphi_0,\varphi_2)$, where $F:\ \Vv\to \Ww$ is
a functor, $\varphi_0:\ l\to F(k)$ is a morphism in $\Ww$, and $\varphi_2:\ F\sq F\Rightarrow
F\circ \ot$ is a natural transformation, satisfying certain properties, we refer to \cite[XI.4]{K}
for detail. A monoidal functor is called strong if $\varphi_0$ and $\varphi_2$ are isomorphisms.

\begin{proposition}\prlabel{1.1}
A monoidal functor  $F:\ \Vv\to \Ww$  induces a bifunctor $F:\ {}_\Vv\Cat\to {}_\Ww\Cat$.
If $F$ is a strong monoidal equivalence of categories, then the induced bifunctor is a
biequivalence.
\end{proposition}

\begin{proof} (sketch).
Let $A$ be a $\Vv$-category, and define $F(A)$ as follows: $F(A)_{x,y}=F(A_{x,y})$. The multiplication
and unit maps are given by the formulas
\begin{eqnarray*}
&&\hspace*{-2cm}
m'_{x,y,z}=F(m_{x,y,z})\circ \varphi_2(A_{x,y},A_{y,z})\\
&:& F(A_{x,y})\ot F(A_{y,z})\to F(A_{x,y}\ot A_{y,z})
\to F(A_{x,z});\\
&&\hspace*{-2cm}\eta'_x=F(\eta_x)\circ\varphi_0:\ l\to F(k)\to F(A_{x,x}).
\end{eqnarray*}
It is straightforward to show that $F(A)$ is a $\Ww$-category.\\
Now let $f:\ A\to B$ be a $\Vv$-functor. $F(f):\ F(A)\to F(B)$ is given by the data
$$F(f)_{x,y}=F(f_{x,y}):\ F(A_{x,y})\to F(B_{f(x),f(y)}).$$
We leave it to the reader to show that $F(f)$ is a $\Ww$-functor.\\
Let $f,g:\ A\to B$ be $\Vv$-functors, and let $\alpha:\ f\rightarrow g$ be a $\Vv$-natural transformation.
$F(\alpha)$ is defined as follows.
$$F(\alpha)_x= F(\alpha_x)\circ \varphi_0:\ l\to F(k)\to F(B_{g(x),f(x)}).$$
$F(\alpha)$ is a $\Ww$-natural transformation, and $F:\ {}_\Vv\Cat\to {}_\Ww\Cat$ is a bifunctor.
Further details are left to the reader.
\end{proof}

Let $\Vv=(\Vv,\ot,k)$ be a monoidal category, and consider its opposite
$\Vv^{\rm op}=(\Vv^{\rm op},\ot^{\rm op},k)$. For later use, we provide a brief description of
$\Vv^{\rm op}$-categories. A $\Vv^{\rm op}$-category consists of a class $X$, $A\in \Vv(X)$
and a collection of morphisms
$$m_{x,y,z}:\ A_{x,z}\to A_{y,z}\ot A_{x,y}~~;~~\eta_x:\ A_{x,x}\to k$$
in $\Vv$. A $\Vv^{\rm op}$-functor $f:\ A\to B$ consists of
$f:\ X\to Y$ together with morphisms $f_{x,y}:\ B_{f(x),f(y)}\to A_{x,y}$ in $\Vv$.
A $\Vv^{\rm op}$-natural transformation $\alpha:\ f\Rightarrow g$ consists of a collection of
morphisms $\alpha_x:\ B_{g(x),f(x)}\to k$ in $\Vv$. We leave it to the reader to formulate all
the necessary axioms that have to be satisfied.

\section{Hopf categories}\selabel{2}
Let $\Vv$ be a strict braided monoidal category, and consider $\Cog(\Vv)$, the category of
coalgebras (or comonoids) and coalgebra morphisms in $\Vv$. $\Cog(\Vv)$ is again a
monoidal category: the tensor product of two coalgebras, resp. two coalgebra morphisms is again
a coalgebra (resp. a coalgebra morphism), and the unit object $k$ of $\Vv$ is a coalgebra.

Now we can consider $\Cog(\Vv)$-categories, that is, categories enriched in $\Cog(\Vv)$.
According to the definitions in \seref{1}, a $\Cog(\Vv)$-category $A$ consists of a class
$|A|=X$, and coalgebras $A_{x,y}$, for all $x,y\in X$, together with
coalgebra morphisms $m_{x,y,z}:\ A_{x,y}\ot A_{y,z}\to A_{x,z}$ and $\eta_x:\ J_{x,x}=ke_{x,x}\to
A_{x,x}$ satisfying (\ref{eq:1.1}-\ref{eq:1.2}).

The definition of a $\Cog(\Vv)$-category can
be restated. Before we do this, we first make the elementary observation that a
coalgebra in $\Vv(X)$ is an object $C\in \Vv(X)$, together with families of morphisms
$\Delta_{x,y}:\ C_{x,y}\to C_{x,y}\ot C_{x,y}$ and $ \varepsilon_{x,y}:\ C_{x,y}\to J_{x,y}=ke_{x,y}$
such that $(C_{x,y},\Delta_{x,y},\varepsilon_{x,y})$ is a coalgebra in $\Vv$, for all $x,y\in X$.
A coalgebra morphism between two coalgebras $C$ and $D$ in $\Vv(X)$ is a morphism
$f:\ C\to D$ in $\Vv(X)$ such that $f_{x,y}$ is a coalgebra map, for all $x,y\in X$.

\begin{proposition}\delabel{2.1}
Let $X$ be a class and let $\Vv$ be a strict braided monoidal category. A $\Cog(\Vv)$-category with underlying
class $X$ is an object in $\Vv(X)$ which has the structure of $\Vv$-category and of a coalgebra in $\Vv(X)$
such that the morphisms $\Delta_{x,y}$ and $\varepsilon_{x,y}$ define $\Vv$-$X$-functors
$\Delta:\ A\to A\bullet A$ and $\varepsilon:\ A\to J$.
\end{proposition}

\begin{proof}
Assume that $A$ is a $\Vv$-category and a coalgebra in $\Vv(X)$, and consider the following diagrams in $\Vv$.
\begin{equation}\eqlabel{2.1.1}
\xymatrix{
A_{x,y}\ot A_{y,z}\ar[d]_{\Delta_{x,y}\ot\Delta_{y,z}}\ar[rr]^{m_{x,y,z}}&&A_{x,z}\ar[d]^{\Delta_{x,z}}\\
A_{x,y}\ot A_{x,y}\ot A_{y,z}\ot A_{y,z}
\ar[rr]^(.60){m^{A\bullet A}_{x,y,z}}&&
A_{x,z}\ot A_{x,z}},
\end{equation}
\begin{equation}\eqlabel{2.1.2}
\xymatrix{ke_{x,x}\ar[drr]_{\eta_x\ot \eta_x}\ar[rr]^{\eta_x}&&A_{x,x}\ar[d]^{\Delta_{x,x}}\\
&&A_{x,x}\ot A_{x,x}},
\end{equation}
\begin{equation}\eqlabel{2.1.3}
\xymatrix{
A_{x,y}\ot A_{y,z}\ar[d]_{\varepsilon_{x,y}\ot\varepsilon_{y,z}}\ar[rr]^{m_{x,y,z}}&&
A_{x,z}\ar[d]^{\varepsilon_{x,z}}\\
ke_{x,y}\ot ke_{y,z}\ar[rr]^{=}&&ke_{x,z}},
\end{equation}
and
\begin{equation}\eqlabel{2.1.4}
\xymatrix{ke_{x,x}\ar[rr]^{\eta_x}\ar[drr]_{=}&&A_{x,x}\ar[d]^{\varepsilon_{x,x}}\\ &&ke_{x,x}}.
\end{equation}
$\Delta$ is a $\Vv$-$X$-functor if and only if the diagrams (\ref{eq:2.1.1}) and (\ref{eq:2.1.2}) commute,
for all $x,y,z\in X$.
$\varepsilon$ is a $\Vv$-$X$-functor if and only if the diagrams \equref{2.1.3} and \equref{2.1.4}
commute, for all $x,y,z\in X$. $m_{x,y,z}$ is a coalgebra map if and only if \equref{2.1.1} and \equref{2.1.3}
commute, and $\eta_x$ is a coalgebra map if and only if \equref{2.1.2} and \equref{2.1.4}
commute.
\end{proof}

Observe that $\Cog(\Vv)$-categories with one object correspond to bialgebras in $\Vv$.
It follows from the results in \seref{1} that $\Cog(\Vv)$-categories can be organized into a 2-category
${}_{\Cog(\Vv)}\dul{\rm Cat}$.
In particular, a $\Cog(\Vv)$-functor between two $\Cog(\Vv)$-categories $A$ and $B$ is a
$\Vv$-functor $f:\ A\to B$ such that every $f_{x,y}:\ A_{x,y}\to B_{x,y}$ is a morphism of coalgebras.
For a fixed class $X$, $\Cog(\Vv)$-categories with underlying class $X$ can be organized into
a 2-category ${}_{\Cog(\Vv)}\dul{\rm Cat}(X)$. A $\Cog(\Vv)$-natural transformation
between two $\Cog(\Vv)$-functors $f,~g:\ A\to B$ consists of grouplike elements $\alpha_x\in
B_{x,x}$ satisfying \equref{1.3}.\\

Let $A$ be a $\Vv$-category, and consider its opposite $A^{\rm op}$ in $\Vv(X)$. $A^{\rm op}$
is also a $\Vv$-category, with multiplication morphisms
$$m^{\rm op}_{x,y,z}=m_{z,y,x}\circ c_{A_{y,x},A_{x,y}}:\
A^{\rm op}_{x,y}\ot A^{\rm op}_{y,z}=A_{y,x}\ot A_{z,y}\to A^{\rm op}_{x,z}=A_{z,x}$$
and unit morphisms $\eta^{\rm op}_x=\eta_x$. Observe that we need the inverse braiding here,
compare to \cite[1.3]{Takeuchi}.\\

Let $C$ be a coalgebra in $\Vv(X)$. The coopposite coalgebra $C^{\rm cop}$ is equal to
$C$ as an object of $\Vv(X)$, with comultiplication morphisms
$$\Delta^{\rm cop}_{x,y}=c^{-1}_{C_{x,y},C_{x,y}}\circ \Delta_{x,y}:\ C_{x,y}\to C_{x,y}\ot C_{x,y},$$
and counit morphisms $\varepsilon^{\rm cop}_{x,y}=\varepsilon_{x,y}$.

\begin{proposition}\prlabel{opcop}
Let $\Vv$ be a strict braided monoidal category, and let $A$ be a $\Cog(\Vv)$-category.
Then $A^{\rm opcop}$ is also  a $\Cog(\Vv)$-category.
\end{proposition}

\begin{proof}
We have to show that the diagrams (\ref{eq:2.1.1}-\ref{eq:2.1.4}) applied to $A^{\rm opcop}$ commute.
\equref{2.1.1} takes the following form:
\begin{equation}\eqlabel{opcop1}
\xymatrix{
A_{y,x}\ot A_{z,y}\ar[d]_{\Delta_{y,x}^{\rm cop}\ot \Delta_{z,y}^{\rm cop}}\ar[rr]^{m^{\rm op}_{x,y,z}}
&&A_{z,x}\ar[d]^{\Delta_{z,x}^{\rm cop}}\\
A_{y,x}\ot A_{y,x}\ot A_{z,y}\ot A_{z,y}\ar[rr]^(.61){m_{A\bullet A,x,y,z}^{\rm op}}
&&A_{z,x}\ot A_{z,x}}
\end{equation}
From the axioms for a braiding $c$, we have the following formula, for all $A,B,C,D\in \Vv$:
\begin{equation}\eqlabel{opcop2}
c_{A\ot B,C\ot D}=(C\ot c_{A,D}\ot B)\circ (c_{A,C}\ot c_{B,D})\circ (A\ot c_{B,C}\ot D).
\end{equation}
The triangle, the squares and the pentangle in the next diagram all commute: the top square commutes
because $c$ is natural; the pentangle is just \equref{2.1.1}; the bottom right square commutes because
$c^{-1}$ is natural; commutativity of the bottom left square follows from \equref{opcop2}. We deleted
the indices in the morphisms in the diagram; they are pretty obvious.
$$\xymatrix{
A_{y,x}\ot A_{z,y}\ar[d]^{\Delta\ot \Delta}\ar[r]^c&
A_{z,y}\ot A_{y,x}\ar[d]^{\Delta\ot \Delta}\ar[r]^{m}&
A_{z,x}\ar[dd]^{\Delta}\\
A_{y,x}\ot A_{y,x}\ot A_{z,y}\ot A_{z,y}\ar[r]^c&
A_{z,y}\ot A_{z,y}\ot A_{y,x}\ot A_{y,x}\ar[d]^{A\ot c\ot A}\ar[dl]_{=}&\\
A_{z,y}\ot A_{z,y}\ot A_{y,x}\ot A_{y,x}\ar[d]^{c^{-1}\ot c^{-1}}&
A_{z,y}\ot A_{y,x}\ot A_{z,y}\ot A_{y,x}\ar[l]_{A\ot c^{-1}\ot A}\ar[d]^{c^{-1}}\ar[r]^(.65){m\ot m}&
A_{z,x}\ot A_{z,x}\ar[d]^{c^{-1}}\\
A_{z,y}\ot A_{z,y}\ot A_{y,x}\ot A_{y,x}\ar[r]^{A\ot c^{-1}\ot A}&
A_{z,y}\ot A_{y,x}\ot A_{z,y}\ot A_{y,x}\ar[r]^(.65){m\ot m}&
A_{z,x}\ot A_{z,x}}$$
From the commutativity of the whole diagram, it follows that
\begin{eqnarray*}
&&\hspace*{-15mm}
\Delta_{z,x}^{\rm cop}\circ m^{\rm op}_{x,y,z}=(m_{z,y,x}\ot m_{z,y,x})\circ (A_{z,y}\ot c^{-1}_{A_{y,x},A_{x,y}}
\ot A_{y,x})\\
&\circ& (c^{-1}_{A_{z,y},A_{z,y}}\ot c^{-1}_{A_{y,x},A_{y,x}})\circ  c_{A_{y,x}\ot A_{y,x}, A_{z,y}\ot A_{z,y}}
\circ (\Delta_{y,x}\ot \Delta_{y,x}).
\end{eqnarray*}
The square at the top of the next diagram commutes because $c$ is natural; commutativity of the bottom triangle
follows from \equref{opcop2}.
$$\xymatrix{
A_{y,x}\ot A_{y,x}\ot A_{z,y}\ot A_{z,y}\ar[rr]^c\ar[d]_{c^{-1}\ot c^{-1}}&&
A_{z,y}\ot A_{z,y}\ot A_{y,x}\ot A_{y,x}\ar[d]^{c^{-1}\ot c^{-1}}\\
A_{y,x}\ot A_{y,x}\ot A_{z,y}\ot A_{z,y}\ar[rr]^c\ar[d]_{A\ot c\ot A}&&
A_{z,y}\ot A_{z,y}\ot A_{y,x}\ot A_{y,x}\ar[d]^{A\ot c^{-1}\ot A}\\
A_{y,x}\ot A_{z,y}\ot A_{y,x}\ot A_{z,y}\ar[rr]^{c\ot c}&&
A_{z,y}\ot A_{y,x}\ot A_{z,y}\ot A_{y,x}}$$
It follows that \equref{opcop1} commutes. The commutativity of the three other diagrams is obvious.
\end{proof}

\prref{opcop} generalizes the fact that the opposite-cooposite of a bialgebra is again a bialgebra:
take $X$ a singleton. We refer to Sweedler \cite{Sweedler} for the case where $\Vv$ is the category
of vector spaces, and to \cite[1.6]{Takeuchi} for the case where $\Vv$ is an arbitrary braided monoidal category.

\begin{definition}\delabel{2.2}
A Hopf $\Vv$-category is a $\Cog(\Vv)$-category $A$ together with a morphism $S:\ A\to A^{\rm op}$ in $\Vv(X)$
such that 
\begin{eqnarray}
m_{x,y,x}\circ (A_{x,y}\ot S_{x,y})\circ \Delta_{x,y}&=& \eta_x\circ \varepsilon_{x,y}:\ A_{x,y}\to A_{x,x}\eqlabel{2.2.1};\\
m_{y,x,y}\circ ( S_{x,y}\ot A_{x,y})\circ \Delta_{x,y}&=& \eta_y\circ \varepsilon_{x,y}:\ A_{x,y}\to A_{y,y}\eqlabel{2.2.2},
\end{eqnarray}
for all $x,y\in X$.
\end{definition}

Observe that a Hopf $\Vv$-category with one object is a Hopf algebra in $\Vv$. If $\Vv=\Mm_k$,
then a Hopf $\Vv$-category is also termed a $k$-linear Hopf category.

\begin{example}\exlabel{2.11} {\bf Sets.}\\
Let $\Vv=(\Sets,\times, \{*\})$. We have seen above that a $\Vv$-category is an ordinary category.
It is well-known that every set $G$ is in a unique way a coalgebra in $\Sets$: the comultiplication is
the diagonal map $G\to G\times G$, sending $g$ to $(g,g)$. The counit is the unique map $G\to \{*\}$.
This means that the categories $\Sets$ and $\Cog(\Sets)$ are identical, and therefore the
same is true for the 2-categories $\Cat={}_{\Sets}\Cat$ and ${}_{\Cog(\Sets)}\Cat$.\\
Now let us investigate Hopf categories. Assume that $G$ is a Hopf category. For all $x,y\in X=|G|$, we have
a map $S_{x,y}:\ G_{x,y}\to G_{y,x}$, satisfying (\ref{eq:2.2.1}-\ref{eq:2.2.2}). Take $a\in G_{x,y}$, this means
that $a:\ y\to x$ is a morphism in $G$. It is easily checked that \equref{2.2.1} implies that $aS_{x,y}(a)=1_x$
and that \equref{2.2.2} implies that $S_{x,y}(a)a=1_y$. This shows that every morphism of $G$ is
invertible, hence $G$ is a
groupoid. Conversely, it is easy to show that a groupoid is a Hopf category.
\end{example}

\begin{proposition}\prlabel{2.12}
Let $\Vv=(\Sets,\times, \{*\})$. Then a Hopf $\Vv$-category is the same thing as a groupoid.
\end{proposition}

\begin{lemma}\lelabel{2.13}
Let $A$ be a Hopf $\Vv$-category. Then the following statements hold, for all $x,y,z\in X$:
\begin{eqnarray}
S_{x,z}\circ m_{x,y,z}&=&m_{z,y,x}\circ (S_{y,z}\ot S_{x,y})\circ c_{A_{x,y},A_{y,z}};\eqlabel{2.13.1}\\
\Delta_{y,x}\circ S_{x,y}&=&c_{A_{y,x},A_{y,x}}\circ (S_{x,y}\ot S_{x,y})\circ \Delta_{x,y}.\eqlabel{2.13.2}
\end{eqnarray}
\end{lemma}

\begin{proof}
In order to make our computations more transparant, we introduce some notation. $A_{x,y}\ot A_{y,z}$ is
a coalgebra, with comultiplication
$$\Delta_{x,y,z}=(A_{x,y}\ot c_{A_{x,y},A_{y,z}}\ot A_{y,z})\circ (\Delta_{x,y}\ot \Delta_{y,z})$$
and counit $\varepsilon_{x,y,z}=\varepsilon_{x,y}\ot \varepsilon_{y,z}$. \equref{2.1.1} can be restated as
\begin{equation}\eqlabel{2.13.3}
\Delta_{x,z}\circ m_{x,y,z}= (m_{x,y,z}\ot m_{x,y,z})\circ \Delta_{x,y,z}.
\end{equation}
The coassociativity of $\Delta_{x,y,z}$ is expressed by the formula
\begin{equation}\eqlabel{2.13.4}
\Delta^2_{x,y,z}=(\Delta_{x,y,z}\ot A_{x,y}\ot A_{y,z})\circ \Delta_{x,y,z}=(A_{x,y}\ot A_{y,z}\ot \Delta_{x,y,z})\circ \Delta_{x,y,z}.
\end{equation}
Now consider the morphisms $f,g,h:\ A_{x,y}\ot A_{y,z}\to Z_{z,x}$ given by the formulas
\begin{eqnarray*}
f&=& m_{z,y,x}\circ (S_{y,z}\ot S_{x,y})\circ c_{A_{x,y},A_{y,z}};\\
g&=& S_{x,z}\circ m_{x,y,z};\\
h&=&m^3_{z,x,y,z,x}\circ (f\ot A_{x,y}\ot A_{y,z}\ot g)\circ \Delta^2_{x,y,z}.
\end{eqnarray*}
We compute that
\begin{eqnarray*}
&&\hspace*{-1cm}
m^2_{x,y,z,x}\circ (A_{x,y}\ot A_{y,z}\ot g)\circ \Delta_{x,y,z}\\
&=& m_{x,z,x}\circ (A_{x,z}\ot S_{x,z})\circ (m_{x,y,z}\ot m_{x,y,z})\circ \Delta_{x,y,z}\\
&\equal{\equref{2.13.3}}&m_{x,z,x}\circ (A_{x,z}\ot S_{x,z})\circ \Delta_{x,z}\circ m_{x,y,z}\\
&\equal{\equref{2.2.1}}&\eta_x\circ \varepsilon_{x,z}\circ m_{x,y,z}\equal{\equref{2.1.3}}\eta_x\circ \varepsilon_{x,y,z},
\end{eqnarray*}
and
$$h=m_{z,x,x}\circ (f\ot \eta_x)\circ (A_{x,y}\ot A_{y,z}\circ \varepsilon_{x,y,z})\circ \Delta_{x,y,z}=f.$$
On the other hand, we have that
\begin{eqnarray*}
&&\hspace*{-1cm}
m^2_{z,x,y,z}\circ (f\ot A_{x,y}\ot A_{y,z})\circ \Delta_{x,y,z}\\
&=&m^3_{z,y,x,y,z}\circ (S_{y,z}\ot S_{x,y}\ot A_{x,y}\ot A_{y,z})\circ (c_{A_{x,y},A_{y,z}}\ot A_{x,y}\ot A_{y,z})\\
&&~~~\circ~ (A_{x,y}\ot c_{A_{x,y},A_{y,z}}\ot A_{y,z})\circ (\Delta_{x,y}\ot \Delta_{y,z})\\
&=& m^3_{z,y,x,y,z}\circ (S_{y,z}\ot S_{x,y}\ot A_{x,y}\ot A_{y,z})\\
&&~~~\circ~(c_{A_{x,y}\ot A_{x,y},A_{y,z}}\ot A_{y,z}) \circ (\Delta_{x,y}\ot \Delta_{y,z})\\
&\equal{(*)}&
m^2_{z,y,y,z}\circ (c_{A_{y,y},A_{z,y}}\ot A_{y,z})\circ (m_{y,x,y}\ot A_{z,y}\ot A_{z,y})\\
&&~~~\circ~ (S_{x,y}\ot A_{x,y}\ot S_{y,z}\ot A_{y,z})\circ (\Delta_{x,y}\ot \Delta_{y,z})\\
&\equal{\equref{2.2.2}}&
m^2_{z,y,y,z}\circ (c_{A_{y,y},A_{z,y}}\ot A_{y,z})\circ (\eta_y \ot A_{z,y}\ot A_{z,y})\\
&&~~~\circ~(S_{y,z}\ot A_{y,z})\circ (\varepsilon_{x,y}\ot \Delta_{y,z})\\
&=&m^2_{z,y,y,z}\circ (A_{z,y}\ot \eta_y\ot A_{z,y}) \circ(S_{y,z}\ot A_{y,z})\circ (\varepsilon_{x,y}\ot \Delta_{y,z})\\
&=& m_{z,y,z}\circ (S_{y,z}\ot A_{y,z})\circ \Delta_{y,z}\circ (\varepsilon_{x,y}\ot A_{y,z})\\
&\equal{\equref{2.2.2}}& \eta_z\circ \varepsilon_{y,z} \circ (\varepsilon_{x,y}\ot A_{y,z})= \eta_z\circ \varepsilon_{x,y,z}.
\end{eqnarray*}
At $(*)$, we used the naturality of the braiding $c$, resulting in the commutativity of the diagram
$$\xymatrix{
(A_{x,y}\ot A_{x,y})\ot A_{y,z}\ar[d]_{(S_{x,y}\ot A_{x,y})\ot S_{y,z}}\ar[rr]^{c_{A_{x,y}\ot A_{x,y}, A_{y,z}}}&&
A_{y,z}\ot (A_{x,y}\ot A_{x,y})\ar[d]^{S_{y,z}\ot (S_{x,y}\ot A_{x,y})}\\
(A_{y,x}\ot A_{x,y})\ot A_{z,y}\ar[d]_{m_{y,x,y}\ot A_{z,y}}\ar[rr]^{c_{A_{y,x}\ot A_{x,y}, A_{z,y}}}&&
A_{z,y}\ot (A_{y,x}\ot A_{x,y})\ar[d]^{A_{z,y}\ot m_{y,x,y}}\\
A_{y,y}\ot A_{z,y}\ar[rr]^{c_{A_{y,y},A_{z,y}}}&&
A_{z,y}\ot A_{y,y}}$$
Finally,
\begin{eqnarray*}
f=h&=& m_{z,z,x}\circ ((\eta_z\circ \varepsilon_{x,y,z})\ot g)\circ \Delta_{x,y,z}\\
&=& m_{z,z,x}\circ (\eta_z\ot A_{z,x})\circ g\circ (\varepsilon_{x,y,z})\ot A_{x,y}\ot A_{y,z})\circ \Delta_{x,y,z}=g.
\end{eqnarray*}
This proves formula \equref{2.13.1}. \equref{2.13.2} is proved using similar techniques. Now we consider
$f,g,h:\ A_{x,y}\to A_{y,x}\ot A_{y,x}$ given by the formulas
\begin{eqnarray*}
f&=& c_{A_{y,x},A_{y,x}}\circ (S_{x,y}\ot S_{x,y})\circ \Delta_{x,y};\\
g&=& \Delta_{y,x}\circ S_{x,y};\\
h&=& m^2_{A\bullet A,y,x,y,x}\circ (g\ot A_{x,y}\ot A_{x,y}\ot f)\circ \Delta_{x,y}^3,
\end{eqnarray*}
In the subsequent computations, the coassociativity of $m^{A\bullet A}$ will be used frequently. We first compute that
\begin{eqnarray*}
&&\hspace*{-2cm}
m^{A\bullet A}_{y,x,y}\circ (g\ot A_{x,y}\ot A_{x,y})\circ \Delta_{x,y}^2\\
&=&
m^{A\bullet A}_{y,x,y}\circ (\Delta_{y,x}\ot \Delta_{x,y})\circ (S_{x,y}\ot A_{x,y})\circ \Delta_{x,y}\\
&\equal{\equref{2.1.1}}&
\Delta_{y,y}\circ m_{y,x,y}\circ (S_{x,y}\ot A_{x,y})\circ \Delta_{x,y}\\
&\equal{\equref{2.2.2}}&
\Delta_{y,y}\circ \eta_\circ \varepsilon_{x,y}=\eta^{A\bullet A}_{y}\circ \varepsilon_{x,y}.
\end{eqnarray*}
It follows that
$$h=m^{A\bullet A}_{y,x,y}\circ(\eta^{A\bullet A}_{y}\ot A_{y,x}\ot A_{y,x})\circ (\varepsilon_{x,y}\ot f)\circ \Delta_{x,y}=f.$$
Now
\begin{eqnarray*}
&&\hspace*{-1cm}
m^{A\bullet A}_{x,y,x}\circ (A_{x,y}\ot A_{x,y}\ot f)\circ \Delta_{x,y}^2\\
&=&
(m_{x,y,x}\ot m_{x,y,x})\circ (A_{x,y}\ot c_{A_{x,y},A_{y,x}}\ot A_{y,x})\\
&&~~~\circ~
(A_{x,y}\ot A_{x,y}\ot c_{A_{y,x},A_{y,x}})\circ (A_{x,y}\ot A_{x,y}\ot  S_{x,y}\ot S_{x,y})\circ \Delta^3_{x,y}\\
&=&
(m_{x,y,x}\ot m_{x,y,x})\circ (A_{x,y}\ot c_{A_{x,y}\ot A_{y,x},A_{y,x}})\\
&&~~~\circ~(A_{x,y}\ot A_{x,y}\ot  S_{x,y}\ot S_{x,y})\circ \Delta^3_{x,y}\\
&\equal{(x)}&
(m_{x,y,x}\ot m_{x,y,x})\circ (A_{x,y}\ot S_{x,y}\ot A_{x,y}\ot S_{x,y})
\\
&&~~~\circ~(A_{x,y}\ot A_{x,y}\ot  \Delta_{x,y})\circ (A_{x,y}\ot c_{A_{x,y},A_{x,y}})\circ \Delta^2_{x,y}\\
&\equal{\equref{2.2.2}}&
(m_{x,y,x}\ot \eta_x)\circ (A_{x,y}\ot S_{x,y})\circ (A_{x,y}\ot A_{x,y}\ot \varepsilon_{x,y})\\
&&~~~\circ~(A_{x,y}\ot c_{A_{x,y},A_{x,y}})\circ \Delta^2_{x,y}\\
&=& (A_{x,x}\ot \eta_x)\circ m_{x,y,x}\circ (A_{x,y}\ot S_{x,y})\circ (A_{x,y}\ot \varepsilon_{x,y}\ot A_{x,y})\\
&&~~~\circ~(A_{x,y}\ot \Delta_{x,y})\circ\Delta_{x,y}\\
&=& (A_{x,x}\ot \eta_x)\circ m_{x,y,x}\circ (A_{x,y}\ot S_{x,y})\circ \circ\Delta_{x,y}\\
&\equal{\equref{2.2.2}}&
(A_{x,x}\ot \eta_x)\circ \eta_x\circ \varepsilon_{x,y}= (\eta_x\ot \eta_x)\circ \varepsilon_{x,y}.
\end{eqnarray*}
At ($x$), we used the naturality of $c$, resulting in the commutative diagram
$$\xymatrix{
A_{x,y}\ot A_{x,y}\ar[d]_{\Delta_{x,y}\ot A}\ar[rr]^{c_{A_{x,y},A_{x,y}}}&&
A_{x,y}\ot A_{x,y}\ar[d]^{A\ot \Delta_{x,y}}\\
A_{x,y}\ot A_{x,y}\ot A_{x,y}\ar[d]_{A_{x,y}\ot S_{x,y}\ot S_{x,y}}&&
A_{x,y}\ot A_{x,y}\ot A_{x,y}\ar[d]^{S_{x,y}\ot A_{x,y}\ot S_{x,y}}\\
A_{x,y}\ot A_{y,x}\ot A_{y,x}\ar[rr]^{c_{A_{x,y}\ot A_{y,x},A_{y,x}}}&&
A_{y,x}\ot A_{x,y}\ot A_{y,x}}$$
Finally
\begin{eqnarray*}
f=h&=& m^{A\bullet A}_{y,x,y}\circ (g\ot ((\eta_x\ot \eta_x)\circ \varepsilon_{x,y})\circ \Delta_{x,y}\\
&=& m^{A\bullet A}_{y,x,y}\circ (A_{y,x}\ot A_{y,x}\ot \eta_{A\bullet A,x})\circ g \circ (A_{x,y}\ot \varepsilon_{x,y})\circ \Delta_{x,y}=g
\end{eqnarray*}
\end{proof}

\begin{theorem}\thlabel{2.3}
Let $A$ be a Hopf $\Vv$-category. The antipode $S:\ A\to A^{\rm opcop}$ is a $\Cog(\Vv)\hbox{-}X$-functor.
\end{theorem}

\begin{proof}
First of all, we need to verify that every $S_{x,y}$ is a morphism in $\Cc(\Vv)$, that is, $S_{x,y}:\ A_{x,y}\to A_{y,x}^{\rm cop}$
is a morphism of coalgebras. To this end, we need the commutativity of the next two diagrams
$$\xymatrix{
A_{x,y}\ar[d]_{S_{x,y}}\ar[r]^(.40){\Delta_{x,y}}&
A_{x,y}\ot A_{x,y}\ar[d]^{S_{x,y}\ot S_{x,y}}\\
A_{y,x}\ar[r]^(.40){\Delta^{\rm cop}_{y,x}}&
A_{y,x}\ot A_{y,x}}\hspace*{3cm}
\xymatrix{
A_{x,y}\ar[d]_{S_{x,y}}\ar[r]^{\varepsilon_{x,y}}& k\\
A_{y,x}\ar[ru]_{\varepsilon_{y,x}}&}$$
The commutativity of the first diagram follows immediately from \equref{2.13.2}. For the second one, we
proceed as follows:
\begin{eqnarray*}
\varepsilon_{x,y}&\equal{\equref{2.1.4}}&\varepsilon_{x,x}\circ \eta_x\circ \varepsilon_{x,y}
\equal{\equref{2.2.1}}\varepsilon_{x,x}\circ m_{x,y,x}\circ (A_{x,y}\ot S_{x,y})\circ \Delta_{x,y}\\
&\equal{\equref{2.1.3}}& (\varepsilon_{x,y}\ot \varepsilon_{y,x})\circ (A_{x,y}\ot S_{x,y})\circ \Delta_{x,y}
= \varepsilon_{y,x}\circ S_{x,y}\circ (\varepsilon_{x,y}\ot _{x,y})\Delta_{x,y}\\
&=&\varepsilon_{y,x}\circ S_{x,y}.
\end{eqnarray*}
Now we show that $S$ is a $\Cog(\Vv)$-functor. The diagrams \equref{Vfunctor} take the following form
$$\xymatrix{
A_{x,y}\ot A_{y,z}\ar[d]_{S_{x,y}\ot S_{y,z}}\ar[rr]^{m_{x,y,z}}&&
A_{x,z}\ar[d]^{S_{x,z}}\\
A_{y,x}\ot A_{z,y}\ar[rr]^{m^{\rm op}_{x,y,z}}&&A_{z,x}}\hspace*{1cm}
\xymatrix{k\ar[r]^{\eta_x}\ar[dr]_{\eta_x}&A_{x,x}\ar[d]^{S_{x,x}}\\ &A_{x,x}}$$
The commutativity of the first diagram follows from \equref{2.13.1}, after making the observation that
$m^{\rm op}_{x,y,z}=m_{z,y,x}\circ c_{A_{y,x},A_{z,y}}$, and taking into account the formula
$$(S_{y,z}\ot S_{x,y})\circ c_{A_{x,y},A_{y,z}}=c_{A_{y,x},A_{z,y}}\circ (S_{x,y}\ot S_{y,z}),$$
resulting from the naturality of $c$. The commutativity of the second diagram goes as follows:
\begin{eqnarray*}
\eta_x&=&
(\varepsilon_{x,x}\ot A_{x,x})\circ \Delta_x\circ \eta_x\equal{\equref{2.1.2}}
(\varepsilon_{x,x}\ot A_{x,x})\circ (\eta_x\ot \eta_x)\\
&=& (\varepsilon_{x,x}\circ \eta_x)\ot \eta_x= \eta_x\circ \varepsilon_{x,x}\circ \eta_x\\
&\equal{\equref{2.2.1}}&
m_{x,x,x}\circ (A_{x,x}\ot S_{x,x})\circ \Delta_{x,x}\circ \eta_x\\
&\equal{\equref{2.1.2}}& m_{x,x,x}\circ (\eta_x\ot A_{x,x})\circ S_{x,x}\circ \eta_x= S_{x,x}\circ \eta_x.
\end{eqnarray*}
\end{proof}

\begin{proposition}\prlabel{2.4}
Let $A$ be a Hopf $\Vv$-category. For $x,y\in X$, consider the following statements:
\begin{eqnarray}
\eta_y\circ \varepsilon_{y,x}&=&
m_{y,x,y}\circ (A_{y,x}\ot S_{y,x})\circ \Delta_{y,x}^{\rm cop}~;\eqlabel{2.4.1}\\
\eta_x\circ \varepsilon_{y,x}&=&
m_{x,y,x}\circ (S_{y,x}\ot A_{y,x})\circ \Delta_{y,x}^{\rm cop}~;\eqlabel{2.4.2}\\
S_{y,x}\circ S_{x,y}&=&A_{x,y}~;\eqlabel{2.4.3}\\
\eta_x\circ \varepsilon_{x,y}&=&
m_{x,y,x}^{\rm op}\circ (S_{x,y}\ot A_{x,y})\circ \Delta_{x,y}~;\eqlabel{2.4.4}\\
\eta_y\circ \varepsilon_{x,y}&=&
m_{y,x,y}^{\rm op}\circ (A_{x,y}\ot S_{x,y})\circ \Delta_{x,y}.\eqlabel{2.4.5}
\end{eqnarray}
The following implications hold:
$$\xymatrix{
\equref{2.4.1}\ar@{=>}[r]&\equref{2.4.3}\ar@{=>}[r]\ar@{=>}[dr]&\equref{2.4.4}\\
\equref{2.4.2}\ar@{=>}[ur]&&\equref{2.4.5}}$$
\end{proposition}

\begin{proof}
$\ul{\equref{2.4.1}\Rightarrow \equref{2.4.3}}$. This goes in two steps. First we compute that
\begin{eqnarray*}
&&\hspace*{-15mm}
m_{y,x,y}\circ (S_{x,y}\ot (S_{y,x}\circ S_{x,y}))\circ \Delta_{x,y}\\
&=&m_{y,x,y}\circ (A_{y,x}\ot S_{y,x})\circ c^{-1}_{A_{y,x},A_{y,x}}\circ c_{A_{y,x},A_{y,x}}\circ (S_{x,y}\ot S_{x,y})\circ \Delta_{x,y}\\
&\equal{\equref{2.13.2}}&
m_{y,x,y}\circ (A_{y,x}\ot S_{y,x})\circ \Delta_{y,x}^{\rm cop}\circ S_{x,y}
\equal{\equref{2.4.1}}\eta_y\circ \varepsilon_{y,x}\circ S_{x,y}=\eta_y\circ \varepsilon_{x,y}.
\end{eqnarray*}
Then we compute that
$$m^2_{x,y,x,y}\circ (A_{x,y}\ot S_{x,y}\ot (S_{y,x}\circ S_{x,y}))\circ \Delta^2_{x,y}$$
is equal to
$$m_{x,y,y}\circ (A_{x,y}\ot \eta_y)\circ (A_{x,y}\ot\varepsilon_{x,y})\circ \Delta_{x,y}=A_{x,y}$$
and, using  \equref{2.2.1}, to
$$m_{x,x,y}(\eta_x\ot A_{x,y})\circ S_{y,x}\circ S_{x,y}\circ (\varepsilon_{x,y}\ot A_{x,y})\circ \Delta_{x,y}
= S_{y,x}\circ S_{x,y}.$$
$\ul{\equref{2.4.3}\Rightarrow \equref{2.4.4}}$.
\begin{eqnarray*}
&&\hspace*{-15mm}
\eta_x\circ \varepsilon_{x,y}= S_{x,x}\circ \eta_x\circ \varepsilon_{x,y}\equal{\equref{2.2.1}}
S_{x,x}\circ m_{x,y,x}\circ (A_{x,y}\ot S_{x,y})\circ \Delta_{x,y}\\
&\equal{\equref{2.13.1}}&
m_{x,y,x}\circ (S_{y,x}\ot S_{x,y})\circ c_{A_{x,y},A_{y,x}}\circ (A_{x,y}\ot S_{x,y})\circ \Delta_{x,y}\\
&=& m_{x,y,x}\circ c_{A_{y,x},A_{x,y}}\circ (S_{x,y} \ot (S_{y,x}\circ S_{x,y}))\circ \Delta_{x,y}\\
&\equal{\equref{2.4.3}}& m_{x,y,x}^{\rm op}\circ (S_{x,y}\ot A_{x,y})\circ \Delta_{x,y}.
\end{eqnarray*}
The proof of the remaining two implications is similar.
\end{proof}

\begin{corollary}\colabel{2.4}
Suppose that $\Vv$ is a symmetric monoidal category. For a Hopf $\Vv$-category, the following assertions
are equivalent:
\begin{enumerate}
\item \equref{2.4.1} holds, for all $x,y\in X$;
\item \equref{2.4.2} holds, for all $x,y\in X$;
\item $S_{y,x}\circ S_{x,y}=A_{x,y}$, for all $x,y\in X$.
\end{enumerate}
\end{corollary}

\begin{proof}
Using the naturality of $c$ and the fact that $c$ is a symmetry, we obtain that
\begin{eqnarray*}
&&\hspace*{-2cm}
m_{x,y,x}^{\rm op}\circ (S_{x,y}\ot A_{x,y})\circ \Delta_{x,y}\\
&=& m_{x,y,x} \circ c_{A_{y,x},A_{x,y}}\circ (S_{x,y}\ot A_{x,y})\circ \Delta_{x,y}\\
&=& m_{x,y,x} \circ (A_{x,y}\ot S_{x,y})\circ c_{A_{x,y},A_{y,x}}\circ \Delta_{x,y}\\
&=& m_{x,y,x} \circ (A_{x,y}\ot S_{x,y})\circ c^{-1}_{A_{x,y},A_{y,x}}\circ \Delta_{x,y}\\
&=& m_{x,y,x} \circ (A_{x,y}\ot S_{x,y})\circ \Delta^{\rm cop}_{x,y}.
\end{eqnarray*}
This tells us that \equref{2.4.4} considered for $(x,y)\in X\times X$ is equivalent to \equref{2.4.1}
considered for $(y,x)\in X\times X$. The statement now follows easily.
\end{proof}

Let $A$ and $B$ be Hopf $\Vv$-categories. A $\Cog(\Vv)$-functor $f:\ A\to B$ is
called a  Hopf $\Vv$-functor if
\begin{equation}\eqlabel{2.5.1}
S^{B}_{f(x),f(y)}\circ f_{x,y}=f_{y,x}\circ S^{A}_{x,y},
\end{equation}
for all $x,y\in X$.

\begin{proposition}\prlabel{2.5}
Let $A$ and $B$ be Hopf $\Vv$-categories. If $f:\ A\to B$ is a 
$\Cog(\Vv)$-functor, then it is also a Hopf $\Vv$-functor.
\end{proposition}

\begin{proof}
Consider the morphisms $k,g,h:\ A_{x,y}\to F_{f(y),f(x)}$ defined by the formulas
$$k= S_{f(x),f(h)}\circ f_{x,y}~~;~~g=f_{y,x}\circ S_{x,y}~~;~~
h=m^2_{f(y),f(x),f(y),f(x)}\circ (k\ot f_{x,y}\ot g)\circ \Delta_{x,y}^2.$$
We have that
\begin{eqnarray*}
&&\hspace*{-2cm}
m_{f(x),f(y),f(x)}\circ (f_{x,y}\ot g)\circ \Delta_{x,y}\\
&=& 
m_{f(x),f(y),f(x)}\circ (f_{x,y}\ot f_{x,y})\circ (A_{x,y}\ot S_{x,y})\circ \Delta_{x,y}\\
&=&
f_{x,x}\circ m_{x,y,x}\circ  (A_{x,y}\ot S_{x,y})\circ \Delta_{x,y}\\
&\equal{\equref{2.2.1}}& f_{x,x}\circ \eta_x\circ \varepsilon_{x,y}=\eta_{f(x)}\circ \varepsilon_{x,y},
\end{eqnarray*}
hence
$$h=m_{f(y),f(y),f(x)}\circ (B_{f(y),f(x)}\ot \eta_{f(x)})\circ k \circ (A_{x,y}\ot \varepsilon_{x,y})\circ \Delta_{x,y}=k.$$
We also have that
\begin{eqnarray*}
&&\hspace*{-2cm}
m_{f(y),f(x),f(y)}\circ (k\ot f_{x,y})\circ \Delta_{x,y}\\
&=&
m_{f(y),f(x),f(y)}\circ (S_{f(x),f(y)}\ot B_{f(x),f(y)})\circ (f_{x,y}\ot f_{x,y})\circ \Delta_{x,y}\\
&=&
m_{f(y),f(x),f(y)}\circ (S_{f(x),f(y)}\ot B_{f(x),f(y)})\circ \Delta_{f(x),f(y)}\circ f_{x,y}\\
&\equal{\equref{2.2.2}}&
\eta_{f(y)}\circ \varepsilon_{f(x),f(y)}\circ f_{x,y}=\eta_{f(y)}\circ \varepsilon_{x,y},
\end{eqnarray*}
so that
$$k=h=m_{f(y),f(y),f(x)}\circ (\eta_{f(y)}\ot B_{f(y),f(x)})\circ g \circ (\varepsilon_{x,y}\ot A_{x,y})\circ \Delta_{x,y}=g.$$
\end{proof}

We introduce ${}_{\Vv}\dul{\rm HopfCat}$ as the full 2-subcategory
of ${}_{\Cog(\Vv)}\Cat$, with Hopf $\Vv$-categories as 0-cells. For two Hopf
$\Vv$-categories $A$ and $B$, the category of morphisms $A\to B$ in
${}_{\Vv}\dul{\rm HopfCat}$ coincides with the category of morphisms $A\to B$ in
${}_{\Cog(\Vv)}\Cat$. Thus 1-cells are Hopf $\Vv$-functors (in view of \prref{2.5})
and 2-cells are $\Cog(\Vv)$-natural transformations.

\begin{proposition}\prlabel{2.6}
Let $F:\ \Vv\to \Ww$ be a strong monoidal functor. $F$ induces bifunctors
$F:\ {}_{\Cog(\Vv)}\Cat\to {}_{\Cog(\Ww)}\Cat$ and
${}_{\Vv}\dul{\rm HopfCat}\to {}_{\Ww}\dul{\rm HopfCat}$.
\end{proposition}

\begin{proof}
$F$ induces a strong monoidal functor $F:\ \Cog(\Vv)\to \Cog(\Ww)$. For a
$\Vv$-coalgebra $C$, $F(C)$ is a $\Ww$-coalgebra. The comultiplication is
$\varphi_2^{-1}\circ F(\Delta):\ F(C)\to F(C)\ot F(C)\to F(C\ot C)$, and the counit
is $\varphi_0^{-1}\circ F(\varepsilon):\ F(C)\to F(k)\to l$.\\
Now apply \prref{1.1} to  $F:\ \Cog(\Vv)\to \Cog(\Ww)$. We obtain a bifunctor
$F:\ {}_{\Cog(\Vv)}\Cat\to {}_{\Cog(\Ww)}\Cat$. For a ${}_{\Cog(\Vv)}$-category $A$,
we have that $F(A)_{x,y}=F(A_{x,y})$, with multiplication maps
$$F(m_{x,y,z})\circ \varphi_2:\ F(A_{x,y})\ot F(A_{y,z})\to F(A_{x,y}\ot A_{y,z})\to F(A_{x,y})$$
and unit maps $F(\eta_x)\circ \varphi_0:\ l\to F(k)\to F(A)$.\\
Now let $A$ be a Hopf $\Vv$-category. We claim that the maps $F(S_{x,y}):\ F(A_{x,y})\to F(A_{y,x})$
define an antipode on $F(A)$. Let us show that \equref{2.2.1} is satisfied. Using the fact that
$\varphi_2$ is natural, we obtain that
\begin{eqnarray*}
&&\hspace*{-2cm}
F(m_{x,y,x})\circ\varphi_2\circ (F(A_{x,y})\ot F(S_{x,y}))\circ \varphi_2^{-1}\circ F(\Delta_{x,y})\\
&=& F(m_{x,y,x})\circ F(A_{x,y}\ot S_{x,y}) \circ \varphi_2\circ \varphi_2^{-1}\circ F(\Delta_{x,y})\\
&=& F(m_{x,y,x}\circ (A_{x,y}\ot S_{x,y})\circ \Delta_{x,y})\\
&\equal{\equref{2.2.1}} & F(\eta_x\circ \varepsilon_{x,y})=F(\eta_x)\circ \varphi_0\circ \varphi_0^{-1}\circ F(\varepsilon_{x,y}),
\end{eqnarray*}
as needed. The proof of \equref{2.2.2} is similar.
\end{proof}

\begin{example}\exlabel{2.7}
Consider the linearization functor $L:\ \Sets\to \Mm_k$. It is well-known that $L$ is
strong monoidal, so, by \prref{2.6}, it sends Hopf categories (which are groupoids, see \prref{4.2})
to $k$-linear Hopf categories. More precisely, consider a groupoid $G$, and let
$G_{x,y}$ be the set of maps from $y$ to $x$. Then $L(G)=A$ is defined as follows:
$$A_{x,y}=kG_{x,y}.$$
The multiplication is the obvious one: the multiplication on $G$ is extended linearly.
$kG_{x,y}$ has the structure of grouplike coalgebra: $\Delta_{x,y}(g)=
g\ot g$ and $\varepsilon_{x,y}(g)=1$ for $g\in G_{x,y}$. The antipode is given by
the formula $S_{x,y}(g)=g^{-1}\in G_{y,x}$.
\end{example}

\section{The representation category}\selabel{3}
\begin{definition}\delabel{3.1}
Let $A$ be a $\Vv$-category. A left $A$-module is an object $M$ in $\Vv(X)$ together with a family of morphisms
$$\psi=
\psi_{x,y,z}:\ A_{x,y}\ot M_{y,z}\to M_{x,z}$$
in $\Vv$ such that the following associativity and unit conditions hold:
\begin{eqnarray}
\psi_{x,y,u}\circ (A_{x,y}\ot \psi_{y,z,u})&=&\psi_{x,z,u}\circ (m_{x,y,z}\ot M_{z,u});\eqlabel{3.1.1}\\
\psi_{x,x,y}\circ (\eta_x\ot M_{x,y})&=&M_{x,y}.\eqlabel{3.1.2}
\end{eqnarray}
Let $M$ and $N$ be left $A$-modules. A morphism $\varphi:\ M\to N$ in $\Vv(X)$ is called left $A$-linear
if
\begin{equation}\eqlabel{3.1.3}
\varphi_{x,z}\circ \psi_{x,y,z}= \psi_{x,y,z}\circ (A_{x,y}\ot \varphi_{y,z}):~~A_{x,y}\ot M_{y,z}\to N_{x,z},
\end{equation}
for all $x,y,z\in X$.
\end{definition}

${}_A\Vv(X)$ will denote the category of left $A$-modules and left $A$-linear morphisms. Right $A$-modules
and $(A,B)$-bimodules are defined in a similar way, and they form categories $\Vv(X)_A$ and ${}_A\Vv(X)_B$.

\begin{proposition}\prlabel{3.2}
Let $A$ be a $\Cog(\Vv)$category. Then there is a monoidal structure on ${}_A\Vv(X)$ such that the forgetful functor
${}_A\Vv(X)\to \Vv(X)$ is monoidal.
\end{proposition}

\begin{proof}
Let $M$ and $N$ be left $A$-modules. We have a left $A$-action on $M\ot N$ as follows:
\begin{eqnarray*}
&&\hspace*{-5cm}
(\psi_{x,y,z}\ot\psi_{x,y,z})\circ (A_{x,y} \ot c_{A_{x,y}, M_{y,z}}\ot N_{y,z})\circ (\Delta_{x,y}\ot M_{y,z}\ot N_{y,z}):\\
A_{x,y} \ot M_{y,z}\ot  N_{y,z}&\rightarrow&A_{x,y} \ot A_{x,y} \ot M_{y,z}\ot  N_{y,z}\\
&\rightarrow& A_{x,y} \ot M_{y,z}\ot A_{x,y} \ot N_{y,z}\\
&\rightarrow&  M_{x,z}\ot N_{x,z} = (M\ot N)_{x,z}.
\end{eqnarray*}
$J$ is a left $H$-module with structure morphisms
$$\varepsilon_{x,y}\ot ke_{y,z}:\ A_{x,y}\ot ke_{y,z}\to ke_{x,y}\ot ke_{y,z}= ke_{x,z}.$$
Verification of all the other details is left to the reader.
\end{proof}

\section{Duality}\selabel{3b}
\subsection{Dual $\Vv$-categories}
The notion of $\Vv$-category can be dualized. A dual $\Vv$-category $C$ consists of a class $|C|=X$
and $C\in \Vv(X)$ together with two classes of morphisms in $\Vv$, namely
$$\Delta_{x,y,z}:\ C_{x,z}\to C_{x,y}\ot C_{y,z}~~{\rm and}~~\varepsilon_x:\ C_{x,x}\to k,$$
satisfying the following coassociativity and counit conditions
\begin{eqnarray*}
&&(\Delta_{x,y,z}\ot C_{z,u})\circ \Delta_{x,z,u}= (C_{x,y}\ot \Delta_{y,z,u})\circ \Delta_{x,y,u};\\
&&(\varepsilon_x\ot C_{x,y})\circ \Delta_{x,x,y}= (C_{x,y}\ot \varepsilon_y)\circ \Delta_{x,y,y}.
\end{eqnarray*}
Dual $\Vv$-categories can be organized into a 2-category ${}^\Vv\Cat$. A 1-cell $f:\ C\to D$
between two dual $\Vv$-categories $C$ and $D$ is a dual $\Vv$-functor, and consists of the following data. For each
$x\in X=|C|$, we have $f(x)\in Y=|D|$, and for each $x,y\in X$, the morphisms $f_{x,y}:\ D_{f(x),f(y)}\to C_{x,y}$ such that
\begin{eqnarray*}
&& (f_{x,y}\ot f_{y,z})\circ \Delta_{f(x),f(y),f(z)}=\Delta_{x,y,z}\circ f_{x,z};\\
&&\varepsilon_{f(x)}=\varepsilon_x\circ f_{x,x}.
\end{eqnarray*}
Let $f,g:\ C\to D$ be dual $\Vv$-functors. A dual $\Vv$-natural transformation $\alpha:\ f\Rightarrow g$
consists of morphisms $\alpha_x:\ D_{f(x),g(x)}\to k$ in $\Vv$ such that
$$(f_{x,y}\ot \alpha_y)\circ \Delta_{f(x),f(y),g(y)}=(\alpha_x \ot g_{x,y})\circ \Delta_{f(x),g(x),g(y)},$$
for all $x,y\in X$. Dual $\Vv$-natural transformations are the 2-cells in ${}^\Vv\Cat$.\\
The composition of 1-cells goes as follows. Let $f:\ C\to D$ and $g:\ D\to E$ be
dual $\Vv$-functors. $g\circ f$ is defined by the formulas
$$(g\circ f)_{x,y}=f_{x,y}\circ g_{f(x),f(y)}:\ E_{(g\circ f)(x),(g\circ f)(y)}\to C_{x,y}.$$
Now let $f':\ C\to D$ and $g':\ D\to E$ be two more
dual $\Vv$-functors, and let $\alpha:\ f\Rightarrow f'$ and $\beta:\ g\Rightarrow g'$ be dual $\Vv$-natural transformations.
$\alpha*\beta:\ g\circ f\Rightarrow g'\circ f'$ is defined by the formulas
\begin{eqnarray*}
(\alpha*\beta)_x & = & \bigl(\beta_{f(x)}\ot (\alpha_x\circ g'_{f(x),f'(x)})\bigr)\circ \Delta_{(g\circ f)(x),(g'\circ f)(x),(g'\circ f')(x)}\\
& = & \bigl( (\alpha_x \circ g_{f(x),f'(x)} )\otimes  \beta_{f'(x)} \bigr) \circ \Delta_{(g\circ f)(x),(g\circ f')(x),(g'\circ f')(x)}
\end{eqnarray*}
Now let $f,g,h:\ C\to D$ be dual $\Vv$-functors, and let $\alpha:\ f\Rightarrow g$, $\beta:\ g\Rightarrow h$
be dual $\Vv$-natural transformations. The vertical composition $\beta\circ\alpha:\ f\Rightarrow h$ is the following:
$$(\beta\circ\alpha)_x=(\alpha_x\ot \beta_x)\circ \Delta_{f(x),g(x),h(x)}:\ B_{f(x),h(x)}\to k.$$

Let $\Vv^{\rm op}=(\Vv^{\rm op},\ot^{\rm op},k)$ be the opposite of the monoidal category $\Vv$. Recall that
$\Hom_{\Vv^{\rm op}}(M,N)=\Hom_{\Vv}(N,M)$, and that the opposite tensor product $\ot$ is given by
$M \ot^{\rm op} N= N\ot M$ and $f \ot^{\rm op} g=g\ot f$.

\begin{proposition}\prlabel{3b.1}
Let $\Vv$ be a strict monoidal category. Then the 2-categories  ${}^{\Vv}\Cat$ and ${}_{\Vv^{\rm op}}\Cat$
are 2-isomorphic.
\end{proposition}

\begin{proof} (Sketch)
We will define a 2-functor $F:\ {}^{\Vv}\Cat\to {}_{\Vv^{\rm op}}\Cat$. Take a dual $\Vv$-category $C$,
with underlying class $X$,
and consider $A=C^{\rm op}$ in $\Vv(X)$. We have $\Vv$-morphisms
$$\Delta_{x,y,z}:\ C_{x,z}=A_{z,x}\to C_{x,y}\ot C_{y,z}=A_{z,y}\ot^{\rm op} A_{y,x},$$
and $\Vv^{\rm op}$-morphisms
$$m_{z,y,x}=\Delta_{x,y,z}:\ A_{z,y}\ot^{\rm op} A_{y,x}\to A_{z,x}.$$
Also $\eta_x=\varepsilon_x:\ k\to A_{x,x}=C_{x,x}$ is a $\Vv^{\rm op}$-morphism, and straightforward
computations show that this makes $A$ a $\Vv^{\rm op}$-category. We define $F(C)=A$.\\
Let $f:\ C\to D$ be a dual $\Vv$-functor, and let $F(D)=B$. For all $x,y\in X$, we have $\Vv$-morphisms
$$f_{x,y}:\ D_{f(x),f(y)}=B_{f(y),f(x)}\to C_{x,y}=A_{y,x}.$$
For all $x,y\in X$, let $g(x)=f(x)$ and $g_{y,x}=f_{x,y}$. Then
$g_{y,x}:\ A_{y,x}\to B_{f(y),f(x)}$ is a $\Vv^{\rm op}$-morphism, and standard arguments tell us
that $g:\ A\to B$ is a $\Vv^{\rm op}$-functor, and we define $F(f)=g$.\\
Finally let $f,f':\ C\to D$ be dual $\Vv$-functors and let $\alpha:\ f\Rightarrow f'$ be a dual $\Vv$-natural
transformation. For every $x\in X$, we have a $\Vv$-morphism $\alpha_x:\
B_{f'(x),f(x)}=D_{f(x),f'(x)}\to k$, and therefore a $\Vv^{\rm op}$-morphism $\alpha_x:\ k\to B_{f'(x),f(x)}
=B_{g'(x),g(x)}$. We leave it to the reader to show that this defines a $\Vv^{\rm op}$-natural
transformation $\alpha:\ g=F(f)\Rightarrow g'=F(f')$. We define $F(\alpha)=\alpha$. Standard computations
show that $F$ is a 2-functor. The inverse of $F$ is defined in a similar way.
\end{proof}

A dual $\Vv$-category with underlying class $X$ is called a dual $\Vv$-$X$-category. A dual $\Vv$-functor
$f$ between two dual $\Vv$-$X$-categories is called a dual $\Vv$-$X$-functor if $f(x)=x$, for all $x\in X$.
${}^{\Vv}\Cat(X)$ is the subcategory of ${}^{\Vv}\Cat$, consisting of dual $\Vv$-$X$-categories,
dual $\Vv$-$X$-functors and dual $\Vv$-natural transformations. As an immediate corollary of
\prref{3b.1}, we have the following result.

\begin{corollary}\colabel{3b.2}
Let $X$ be a class, and let $\Vv$ be a strict monoidal category. Then the 2-categories 
${}_{\Vv^{\rm op}}\Cat(X)$ and ${}^{\Vv}\Cat(X)$
are 2-isomorphic.
\end{corollary}

If $X$ is a singleton, then the objects in ${}^{\Vv}\Cat(X)$ are $\Vv$-coalgebras. Deleting the non-unit 2-cells in
${}^{\Vv}\Cat(X)$, we obtain $\Cog(\Vv)^{\rm op}$, the opposite of the category of coalgebras. 

\subsection{Modules versus comodules}
We now consider $\Vv=(\Mm^{\rm f}_k,\ot,k)$, the category of finitely generated
projective modules over a commutative ring $k$, and its opposite $\Vv^{\rm op}=(\Mm^{\rm fop}_k,\ot^{\rm op},k)$.
It is well-known that the functor $(-)^*:\ \Mm^{\rm f}_k\to \Mm^{\rm fop}_k$ taking a module $M$ to its dual
$M^*=\Hom(M,k)$ is an equivalence of categories. Moreover, we have a strong monoidal functor
$$((-)^*, \varphi_0,\varphi_2):\ (\Mm^{\rm f}_k,\ot,k)\to (\Mm^{\rm fop}_k,\ot^{\rm op},k).$$
Let
$\varphi_0:\ k\to (k)^*=k$ be the identity map. We now construct a natural isomorphism
$$\varphi_2:\ \ot^{\rm op}\circ ((-)^*,(-)^*)\Rightarrow (-)^*\circ \ot.$$
For two finitely generated projective $k$-modules $M$ and $N$, we need an isomorphism
$$\varphi_2(M,N):\ M^*\ot^{\rm op} N^*\to (M\ot N)^*$$
in $\Mm^{\rm fop}_k$, or, equivalently, an isomorphism
$$\varphi_2(M,N):\  (M\ot N)^*\to N^*\ot M^*$$
in $\Mm^{\rm f}_k$. It is well-known that the map
$$\iota:\ N^*\ot M^*\to (M\ot N)^*,~~\lan \iota(n^*\ot m^*),m\ot n\ran=\lan n^*,n\ran \lan m^*,m\ran$$
is invertible, with inverse given by the formula
$$\iota^{-1}(\mu)=\sum_{i,j} \lan \mu,m_i\ot n_j\ran n_j^*\ot m_i^*,$$
where $\sum_i m_i\ot m_i^*$ and $\sum_j n_j\ot n_j^*$ are the finite dual bases of $M$ and $N$.
We now define $\varphi_2(M,N)$ as the inverse of $\iota$.
As $((-)^*,\varphi_0,\varphi_2)$ is strong monoidal, it follows from \prref{1.1} that we have
a biequivalence between ${}_{\Mm^{\rm f}_k}\Cat$ and ${}_{\Mm^{\rm fop}_k}\Cat$.
Applying \prref{3b.1}, we find that ${}_{\Mm^{\rm fop}_k}\Cat$ is 2-isomorphic to
${}^{\Mm^{\rm f}_k}\Cat$. Combining these two biequivalences, we obtain the following result.

\begin{theorem}\thlabel{3c.1}
Let $k$ be a commutative ring. $(-)^*$ induces a biequivalence
$${}_{\Mm^{\rm f}_k}\Cat\to {}^{\Mm^{\rm f}_k}\Cat.$$
\end{theorem}

Let us describe this biequivalence at the level of 0-cells. Suppose that $A$ is a
$k$-linear category, with all underlying $A_{x,y}$ finitely generated and projective.
 First we have to apply the duality functor
$(-)^*$, sending $A$ to $A^*$, with $(A^*)_{x,y}=A_{x,y}^*$.
 In order to compute the
multiplication and unit maps, we have to apply the construction sketched in the proof
of \prref{1.1}. The multiplication is the following composition in $\Mm_k^{\rm fop}$:
$$m^*_{x,y,z}\circ \varphi_2(A_{x,y},A_{y,z}):\ A^*_{y,z}\ot A^*_{x,y}\to
(A_{x,y}\ot A_{y,z})^*\to A^*_{x,z}.$$
The unit map is $\eta_x^*:\ k\to A^*_{x,x}$ in $\Mm_k^{\rm fop}$. To $A^*$, we apply
the construction performed in the proof of \prref{3b.1}, which sends $A^*$ to $C$,
with $C_{x,y}=A^*_{y,x}$. The comultiplication maps are the following maps in $\Mm_k^{\rm f}$:
$$\Delta_{z,y,x}= \varphi_2(A_{x,y},A_{y,z})\circ m^*_{x,y,z}:\
A^*_{x,z}=C_{z,x}\to A^*_{y,z}\ot A^*_{x,y}=C_{z,y}\ot C_{y,x}.$$
The counit maps are $\varepsilon_x=\eta_x^*:\ C_{x,x}=A^*_{x,x}\to k$.\\
 Let us also give a brief description of the inverse construction. 
Let $(C,\Delta,\varepsilon)$ be a dual $\Mm_k^{\rm f}$-category. We will use the following
Sweedler-Heyneman type notation: for $c\in C_{x,z}$, $\Delta_{x,y,z}(c)=
c_{(1,y)}\ot c_{(2,y)}\in C_{x,y}\ot C_{y,z}$. Let $A\in \Mm_k^{\rm fop}(X)$ be defined
as $A_{x,y}=C^*_{y,x}$. The multiplication map $m_{x,y,z}:\ A_{x,y}\ot A_{y,z}\to
A_{x,z}=C_{z,x}^*$ is defined by the formula
$$\lan ab, c\ran=\lan a,c_{(2,y)}\ran \lan b,c_{(1,y)}\ran.$$
for $a\in A_{x,y}$, $b\in A_{y,z}$, $c\in C_{z,x}$. The unit elements are
$\varepsilon_x\in C_{x,x}^*=A_{x,x}$.\\

Let $C$ be a dual $k$-linear category.
A right $C$-comodule $M$ is an object $M\in \Vv(X)$ together with a family of maps
$$\rho_{x,y,z}:\ M_{x,z}\to M_{x,y}\ot C_{y,z}$$
such that the coassociativity and counit conditions (\ref{eq:3b.1.3}-\ref{eq:3b.1.4}) are satisfied.
For $m\in M_{x,z}$, we will write
$$\rho_{x,y,z}(m)=m_{[0,y]}\ot m_{[1,y]}.$$
For all $m\in M_{x,z}$, we need that
\begin{equation}\eqlabel{3b.1.3}
m_{[0,y][0,u]}\ot m_{[0,y][1,u]}\ot m_{[1,y]}=m_{[0,u]}\ot m_{[1,u](1,y)}\ot m_{[1,u](2,y)},
\end{equation}
in $M_{x,u}\ot C_{u,y}\ot C_{y,z}$, and 
\begin{equation}\eqlabel{3b.1.4}
m_{[0,z]}\varepsilon_z(m_{[1,z]})=m.
\end{equation}

\begin{proposition}\prlabel{3b.5}
Let $k$ be a commutative ring, and
let $C$ be a dual  $k$-linear category,with underlying class $X$, and with all $C_{x,y}$ finitely generated and projective.
Let $A$ be the corresponding 
$k$-linear category. Then the categories $\Mm_{k}^{\rm fop}(X)^C$ and $\Mm_{k}^{\rm f}(X)_A$
are isomorphic.
\end{proposition}

\begin{proof}
Let $M$ be a right $C$-comodule. We have the structure maps
$$\rho_{x,y,z}:\ M_{x,z}\rightarrow M_{x,y}\otimes C_{y,z}$$
Now we claim that $M$ is also a right $A$-module, with structure maps
$$\psi_{x,z,y}:\ M_{x,z}\ot A_{z,y}\to M_{x,y},~~
\psi_{x,z,y}(m\ot a)=ma=\lan a,m_{[1,y]}\ran m_{[0,y]}.$$
Let us first show that this right $A$-action is associative. Take $m\in M_{x,z}$,
$a\in A_{z,y}$ and $b\in A_{y,u}$. Then
\begin{eqnarray*}
(ma)b&=&\lan a,m_{[1,y]}\ran \lan b,m_{[0,y][1,u]}\ran m_{[0,y][0,u]}\\
&\equal{\equref{3b.1.3}}&
\lan a, m_{[1,u](2,y)}\ran \lan b, m_{[1,u](1,y)}\ran m_{[0,u]}\\
&=& \lan ab, m_{[1,u]}\ran m_{[0,u]}=m(ab).
\end{eqnarray*}
Now we prove the unit property. The unit element of $A_{x,x}$ is $\varepsilon_{x}$,
and for all $m\in M_{x,x}$, we have that $m\varepsilon_x=\lan \varepsilon_x,m_{[1,x]}\ran
m_{[0,x]}=m$.\\
Conversely, let $M$ be a right $A$-module. As before, let $\sum_i a_i^{y,z}\ot c_i^{y,z}
\in A_{z,y}\ot C_{y,z}$ be the finite dual basis of $C_{y,z}$. We define a right $C$-coaction
on $M$, via the structure maps
$$\rho_{x,y,z}:\ M_{x,z}\to M_{x,y}\ot C_{y,z},~~\rho_{x,y,z}(m)=\sum_i ma_i^{y,z}\ot c_i^{y,z}.$$
It is straightforward to show that this makes $M$ into a right $C$-comodule.\\
These two constructions are inverses. First we start with a right $C$-coaction on $M$.
The above construction then provides a right $A$-action on $M$, and the a new right
$C$-coaction $\tilde{\rho}$, which coincides with the original $\rho$. Indeed, for all
$m\in M_{x,z}$, we have that
\begin{eqnarray*}
\tilde{\rho}_{x,y,z}(m)&=& \sum_i ma_i^{y,z}\ot c_i^{y,z}=
\sum_i \lan a_i^{y,z},m_{[1,y]}\ran m_{[0,y]}\ot c_i^{y,z}\\
&=& m_{[0,y]}\ot m_{[1,y]}={\rho}_{x,y,z}(m).
\end{eqnarray*}
Now start from a right $A$-action on $M$. Applying the two constructions from above, we arrive
first at a right $C$-coaction on $M$, and then a new right $A$-action that coincides with the
original one: for $m\in M_{x,z}$ and $a\in A_{z,y}$, we have that
$$m\cdot a=\lan a, m_{[1,y]}\ran m_{[0,y]}=\sum_i \lan a, c_i^{y,z}\ran ma_i^{y,z}=ma.$$
\end{proof}

\subsection{Duality between Hopf categories and dual Hopf categories}
$(-)^*$ induces an equivalence of categories
$(-)^*:\ \Cog(\Mm_k^{\rm f})\to \Cog(\Mm_k^{\rm fop})$. Observing that the categories
$\Cog(\Mm_k^{\rm fop})$ and $\Alg(\Mm_k^{\rm f})^{\rm op}$ are isomorphic, we obtain an
equivalence of categories
$$(-)^*:\ \Cog(\Mm_k^{\rm f})\to \Alg(\Mm_k^{\rm f})^{\rm op}.$$
Let us compute the algebra structure on the dual $C^*$ of a coalgebra $C$. The coalgebra structure in
$\Mm_k^{\rm fop}$ is the composition
$$\varphi_2(C,C)^{-1}\circ \Delta^*:\ C^*\to (C\ot C)^*\to C^*\ot C^*,$$
in $\Mm_k^{\rm fop}$ which is the composition
$$m=\Delta^*\circ \iota:\ C^*\ot C^*\to (C\ot C)^*\to C^*.$$
It easily computed that $m$ is the opposite of the convolution product, that is
$m(c^*\ot d^*)=c^*d^*$, with $\lan c^*d^*, c\ran=\lan c^*,c_{(2)}\ran \lan d^*,c_{(1)}\ran$. Now we claim
that we have a strong monoidal equivalence
$$((-)^*, \varphi_0,\varphi_2):\ (\Cog(\Mm_k^{\rm f}),\ot , k)\to (\Alg(\Mm_k^{\rm f})^{\rm op},\ot^{\rm op},k).$$
$\varphi_0$ is again the identity on $k$, and
$$\varphi_2(C,D):\ D^*\ot C^*\to (C\ot D)^*$$
in $\Alg(\Mm_k^{\rm f})^{\rm op}$ is the inverse of the map $\iota$ defined above.
It follows from \prref{1.1} that $(-)^*$ induces a biequivalence
$$(-)^*:\ {}_{\Cog(\Mm_k^{\rm f})}\Cat\to {}_{\Alg(\Mm_k^{\rm f})^{\rm op}}\Cat.$$
We now from \prref{3b.1} that ${}_{\Alg(\Mm_k^{\rm f})^{\rm op}}\Cat$ is 2-isomorphic to
${}^{\Alg(\Mm_k^{\rm f})}\Cat$. Hence we have the following result.

\begin{theorem}\thlabel{3b.3}
Let $k$ be a commutative ring. We have a biequivalence
$${}_{\Cog(\Mm_k^{\rm f})}\Cat\to {}^{\Alg(\Mm_k^{\rm f})}\Cat.$$
\end{theorem} 
 
 For a ${}_{\Cog(\Mm_k^{\rm f})}\Cat$-category $A$, we provide the corresponding
dual ${}^{\Alg(\Mm_k^{\rm f})}\Cat$-category $C$. First we have to apply the duality functor
$(-)^*$, sending $A$ to $A^*$, with $(A^*)_{x,y}=A_{x,y}^*$. Then we apply
the construction performed in the proof of \prref{3b.1}, which sends $A^*$ to $C$,
with $C_{x,y}=A^*_{y,x}$. From \thref{3c.1}, we already know the dual $k$-linear category structure
on $C$. 
Each $C_{x,y}=A_{y,x}^*$
is a $k$-coalgebra, with opposite convolution as multiplication, and $1_{x,y}=\varepsilon_{y,x}$
as unit element. \\
Let us also give a brief description of the inverse construction. 
Let $(C,\Delta,\varepsilon)$ be a dual $\Mm_k^{\rm f}$-category. 
The $k$-linear category structure on $A$ has already been given in the comments following
\thref{3c.1}. Each $A_{x,y}=C^*_{y,x}$ is a $k$-coalgebra
with comultiplication
$$\Delta(a)=\sum_{i,j} \lan a,c_ic_j\ran a_j^*\ot a_i^*,$$
where $\sum_i c_i\ot a_i\in C_{y,x}\ot A_{x,y}$ is the dual basis of $C_{y,x}$.\\

Let $C$ be a dual $\Vv$-category. $C$ is called a dual Hopf $\Vv$-category if
there exist morphisms $S_{x,y}:\ C_{y,x}\to C_{x,y}$ in $\Vv$ such that
\begin{eqnarray}
m_{x,y}\circ (C_{x,y}\ot S_{x,y})\circ \Delta_{x,y,x} &=& \eta_{x,y}\circ \varepsilon_x;\eqlabel{3b.1.9}\\
m_{y,x}\circ (S_{y,x}\ot C_{y,x})\circ \Delta_{x,y,x}&=& \eta_{y,x}\circ \varepsilon_x.\eqlabel{3b.1.10}
\end{eqnarray}

\begin{theorem}\thlabel{3b.4}
Let $k$ be a commutative ring. In the biequivalence from \thref{3b.3},
Hopf $\Mm_k^{\rm f}$-categories correspond to dual Hopf $\Mm_k^{\rm f}$-categories.
\end{theorem}

\begin{proof}
Assume that $C$ is a dual Hopf $\Mm_k^{\rm f}$-category with antipode $S$, and let
$A$ be the corresponding Hopf $\Mm_k^{\rm f}$-category. We claim that $T$ defined by
$$T_{x,y}=S^*_{y,x}:\ A_{x,y}\to A_{y,x}$$
is an antipode for $A$.
We have to show that \equref{2.3.1} holds. The first formula in \equref{2.3.1}
reduces to
$$a_{(1)}T_{x,y}(a_{(2)})=\lan a,1_{y,x}\ran \varepsilon_x,$$
in $A_{x,x}=C_{x,x}^*$, for all $a\in A_{x,y}$. For all $c\in C_{x,x}$, we have that
\begin{eqnarray*}
&&\hspace*{-2cm}
\lan a_{(1)}T_{x,y}(a_{(2)}),c\ran = \lan a_{(1)},c_{(2,y)}\ran \lan T_{x,y}(a_{(2)}),c_{(1,y)}\ran\\
&=& \lan a_{(1)},c_{(2,y)}\ran \lan a_{(2)}, S_{y,x}(c_{(1,y)})\ran\\
&=& \lan a, S_{y,x}(c_{(1,y)})c_{(2,y)}\ran\equal{\equref{3b.1.10}}\lan a,1_{y,x}\ran \lan \varepsilon_x,c\ran.
\end{eqnarray*}
The second formula in \equref{2.3.1} is proved in a similar way.
\end{proof}

\section{Hopf categories and Hopf group (co)algebras}\selabel{4}
Let $(\Vv,\ot,k)$ be a monoidal category. A group graded $\Vv$-algebra consists of a group $G$ together
with a family of objects $A=\{A_\sigma~|~\sigma\in G\}$ in $\Vv$ and morphisms
$$m_{\sigma,\tau}:\ A_\sigma\ot A_\tau\to A_{\sigma\tau}~~;~~\eta:\ k\to A_e$$
in $\Vv$ such that the following associativity and unit properties hold, for all $\sigma,\tau,\rho\in G$:
\begin{eqnarray*}
m_{\sigma\tau,\rho}\circ (m_{\sigma,\tau}\ot A_\rho)&=&m_{\sigma,\tau\rho}\circ (A_\sigma\ot m_{\tau,\rho});\\
m_{e,\sigma}\circ (\eta\ot A_{\sigma})&=&m_{\sigma,e}\circ (A_\sigma\ot \eta)=A_\sigma.
\end{eqnarray*}
Consider the case where $\Vv$ is the category of modules over a commutative ring $k$, and let
$A=\{A_\sigma~|~\sigma\in G\}$ be a graded algebra. Then $A=\oplus_{\sigma\in G} A_{\sigma}$ is
a $G$-graded algebra in the usual sense (see \cite{NVO} for the general theory of graded algebras), and is called a graded algebra in packed form.
Graded algebras can be organized into a 2-category ${}_\Vv\gr$.\\
A 1-cell $f:\ (G,A)\to (H,B)$ consists of a 
a group morphism $f:\ G\to H$ together with a family of morphisms $f_\sigma:\ A_\sigma\to B_{f(\sigma)}$
in $\Vv$ such that $f_{\sigma\tau}\circ m_{\sigma,\tau}=m_{f(\sigma),f(\tau)}\circ (f_\sigma\ot f_\tau)$ and
$f_e\circ\eta=\eta$.\\
Let $f,g:\ (G,A)\to (H,B)$ be 1-cells; a 2-cell $\alpha:\ f\Rightarrow g$ consists of a family of morphisms
$\alpha_\sigma:\ k\to B_{g(\sigma)^{-1}f(\sigma)}$ such that the following diagrams commute:
$$\xymatrix{
A_{\sigma^{-1}\tau}\ar[d]_{\alpha_\sigma\ot f_{\sigma^{-1}\tau}}
\ar[rr]^(.40){g_{\sigma^{-1}\tau}\ot \alpha_\tau}&&
B_{g(\sigma)^{-1}g(\tau)}\ot B_{g(\tau)^{-1}f(\tau)}\ar[d]^{m_{g(\sigma)^{-1}g(\tau),g(\tau)^{-1}f(\tau)}}\\
B_{g(\sigma)^{-1}f(\sigma)}\ot B_{f(\sigma)^{-1}f(\tau)}\ar[rr]^(.60){m_{g(\sigma)^{-1}f(\sigma), f(\sigma)^{-1}f(\tau)}}
&&B_{g(\sigma)^{-1}f(\tau)}
}$$
We have the dual notion of graded coalgebra. A group graded coalgebra in $\Vv$ consists of a group $G$ together
with a family of objects $C=\{C_\sigma~|~\sigma\in C\}$ in $\Vv$ and morphisms
$$\Delta_{\sigma,\tau}:\ C_{\sigma\tau}\to C_\sigma\ot C_\tau~~;~~
\varepsilon:\ C_e\to k$$
such that
\begin{eqnarray*}
(\Delta_{\sigma,\tau}\ot C_\rho)\circ \Delta_{\sigma\tau,\rho}&=&(C_\sigma\ot \Delta_{\tau,\rho})\circ \Delta_{\sigma,\tau\rho}\\
(\varepsilon\ot C_\rho)\circ \Delta_{e,\sigma}&=&(C_\sigma\ot \varepsilon)\circ \Delta_{\sigma,e}=C_\sigma.
\end{eqnarray*}
Let $\Vv=\Mm_k$, and suppose that $G$ is a finite group. If $C$ is a $G$-graded coalgebra, then
$\oplus_{\sigma\in G} C_\sigma$ is a $G$-graded coalgebra in the sense of \cite{NT}.\\
Graded coalgebras can be organized into a 2-category ${}^\Vv\gr$.\\
A 1-cell $f:\ (G,C)\to (H,D)$ is a morphism of graded coalgebras. This consists of a 
a group morphism $f:\ G\to H$ together with a family of morphisms $f_\sigma:\ D_{f(\sigma)}\to C_\sigma$ such that
$(f_\sigma\ot f_\tau)\circ \Delta_{f(\sigma),f(\tau)}=\Delta_{\sigma,\tau}\circ f_{\sigma\tau}$ and $\varepsilon\circ f_e=\varepsilon$.\\
Now let $f,g:\ C\to D$ be 1-cells. A 2-cell $\alpha:\ f\rightarrow g$ consists of a family of morphisms
$\alpha_\sigma:\ D_{f(\sigma)^{-1}g(\sigma)}\to k$ such that
$$(f_{\sigma^{-1}\tau}\ot \alpha_\tau)\circ \Delta_{f(\sigma)^{-1}f(\tau),f(\tau)^{-1}g(\tau)}=(\alpha_\sigma\ot g_{\sigma^{-1}\tau})\circ
\Delta_{f(\sigma)^{-1}g(\sigma),g(\sigma)^{-1}g(\tau)}.$$

\begin{proposition}\prlabel{4.1}
Let $\Vv$ be a strict monoidal category. Then the 2-categories  ${}^{\Vv}\gr$ and ${}_{\Vv^{\rm op}}\gr$
are 2-isomorphic.
\end{proposition}

\begin{proof}
The proof is similar to the proof of \prref{3b.1}. We will describe the 2-functor $F:\ {}^{\Vv}\gr$ and ${}_{\Vv^{\rm op}}\gr$.
Let $(G,C)$ be a graded coalgebra, and let $F(G,C)=(G,A)$, with $A_\sigma=C_{\sigma^{-1}}$. The multiplication map
$m_{\sigma,\tau}:\ A_\sigma\ot^{\rm op} A_\tau\to A_{\sigma\tau}$ in $\Vv^{\rm op}$ is given by
$\Delta_{\tau^{-1},\sigma^{-1}} C_{\tau^{-1}\sigma^{-1}}\to C_{\tau^{-1}}\ot C_{\sigma^{-1}}$ in $\Vv$.\\
Let $f:\ (G,C)\to (H,D)$ be a morphism of graded coalgebras. We define $F(f)=g:\ F(G,C)=(G,A)\to F(H,D)=(H,B)$ as follows:
$g(\sigma)=\sigma$, for all $\sigma\in G$, and $g_\sigma : A_{\sigma}\to B_{f(\sigma)}$ in $\Vv^{\rm op}$ is the map
$f_{\sigma^{-1}}:\ D_{f(\sigma)^{-1}}=B_{f(\sigma)}\to C_{\sigma^{-1}}=A_\sigma$ in $\Vv$.\\
Let $f,f':\ (G,C)\to (H,D)$ be morphisms of graded coalgebras, and let $\alpha:\ f\Rightarrow f'$ be a 2-cell in ${}^{\Vv}\gr$.
We have morphisms $\alpha_\sigma:\ D_{f(\sigma)^{-1}f'(\sigma)}\to k$ in $\Vv$, which are also morphisms
$\alpha_\sigma:\ k\to B_{f'(\sigma)^{-1}f(\sigma)}$ in $\Vv^{\rm op}$, defining a 2-cell $F(f)\Rightarrow F(f')$ in ${}_{\Vv^{\rm op}}\gr$.
\end{proof}

\begin{proposition}\prlabel{4.2}
Let $\Vv$ be a strict monoidal category. We have 2-functors $K:\ {}_{\Vv}\gr\to {}_{\Vv}\Cat$ and $H:\ {}^{\Vv}\gr\to {}^{\Vv}\Cat$.
\end{proposition}

\begin{proof}
Let $A$ be a $G$-graded algebra. We define a $\Vv$-category $K(A)=K(G,A)$ as follows. The underlying class is $G$, and
$K(A)_{\sigma,\tau}=A_{\sigma^{-1}\tau}$. The multiplication maps are
\begin{eqnarray*}
&&\hspace*{-15mm}
m_{\sigma,\rho,\tau}= m_{\sigma^{-1}\rho, \rho^{-1}\tau}\\
&:& K(A)_{\sigma,\rho}=A_{\sigma^{-1}\rho}\ot K(A)_{\rho,\tau}=A_{\rho^{-1}\tau}\to K(A)_{\sigma,\tau}=A_{\sigma^{-1}\tau},
\end{eqnarray*}
and the unit maps are $\eta_x=\eta:\ k\to A_e=A_{\sigma,\sigma}$.\\
Let $f:\ (G,A)\to (H,B)$ be a morphism of graded algebras. $K(f)=g:\ K(G,A)\to K(H,B)$ is then defined as follows.
$g(\sigma)=f(\sigma)$, for all $\sigma\in G$, and $g_{\sigma,\tau}=f_{\sigma^{-1}\tau}:\ K(A)_{\sigma,\tau}=A_{\sigma^{-1}\tau}
\to K(B)_{f(\sigma),f(\tau)}=B_{f(\sigma)^{-1}f(\tau)}$.\\
Now let $\alpha:\ f\Rightarrow f'$ be a 2-cell in ${}_{\Vv}\gr$. We have morphisms $\alpha_\sigma:\ k\to
B_{g(\sigma)^{-1}f(\sigma)}=K(B)_{g(\sigma ),f(\sigma)}$, and these also define a 2-cell $g\Rightarrow g'$ in ${}_{\Vv}\Cat$.\\
The 2-functor $H:\ {}^{\Vv}\gr\to {}^{\Vv}\Cat$ is constructed in a similar way. Let us just mention that, for a $G$-graded
coalgebra $C$, $H(C)_{\sigma,\tau}=C_{\sigma^{-1}\tau}$.
\end{proof}

Let $\Vv$ be a braided (strict) monoidal category. We can consider graded coalgebras in $\Alg(\Vv)$ and graded algebras
in $\Cog(\Vv)$. A graded coalgebra in $\Alg(\Vv)$ is a graded coalgebra $C$ in $\Vv$, such that every $C_\sigma$ is
an algebra in $\Vv$, and the comultiplication and counit morphisms $\Delta_{\sigma,\tau}$ and $\varepsilon$ are
algebra maps. Graded coalgebras in $\Alg(\Vv)$ are known in the literature as semi-Hopf group coalgebras.
They appeared in \cite{T} (see also \cite{T2}), and a systematic algebraic study was initiated in \cite{V}.\\
In a similar way, a graded algebra in $\Cog(\Vv)$ is a graded algebra $A$ in $\Vv$ such that every $A_\sigma$ is a
coalgebra in $\Vv$, and the multiplication and counit morphisms $m_{\sigma,\tau}$ and $\eta$ are coalgebra morphisms.
In the literature, this is also called a semi-Hopf group algebra.\\
This provides us with a new categorical interpretation of semi-Hopf group algebras and coalgebras. We also obtain
that semi-Hopf group algebras (resp. coalgebras) can be organized into a 2-category ${}_{\Cog(\Vv)}\gr$
(resp. ${}^{\Alg(\Vv)}\gr$).
Note that a different interpretation,
where group algebras and coalgebras appear as bialgebras in a suitable symmetric monoidal category was given
by the second author and De Lombaerde in \cite{CL}.\\
Recall that a semi-Hopf group coalgebra $C$ is called a Hopf group coalgebra if there exist morphisms $S_\sigma:\
C_{\sigma^{-1}}\to C_\sigma$ such that
$$m_\sigma\circ (C_\sigma\ot S_\sigma)\circ \Delta_{\sigma,\sigma^{-1}}=
m_\sigma\circ ( S_\sigma\ot C_\sigma)\circ \Delta_{\sigma^{-1},\sigma}=\eta_\sigma\circ \varepsilon.$$
A semi-Hopf group algebra $A$ is called a Hopf group algebra if there exist morphisms $S_\sigma:\
A_{\sigma}\to A_{\sigma^{-1}}$ such that
$$m_{\sigma,\sigma^{-1}}\circ (A_\sigma\ot S_\sigma)\circ \Delta_\sigma=
m_{\sigma^{-1},\sigma}\circ (S_\sigma\ot A_\sigma)\circ \Delta_\sigma=\eta\circ \varepsilon_\sigma.$$

\begin{proposition}\prlabel{4.3}
Let $\Vv$ be a braided strict monoidal category. We have 2-functors $K:\ {}_{\Cog(\Vv)}\gr\to {}_{\Cog(\Vv)}\Cat$ and 
$K:\ {}^{\Alg(\Vv)}\gr\to {}^{\Alg(\Vv)}\Cat$. The first functor sends Hopf group algebras to Hopf $\Vv$-categories,
and the second one sends Hopf group coalgebras to dual Hopf $\Vv$-categories.
\end{proposition}

\begin{proof}
The first statement is an immediate corollary of \prref{4.2}. The proof of the second statement is straightforward.
Let $A$ be a Hopf group algebra. $K(S)_{\sigma,\tau}=S_{\sigma^{-1}\tau}:\ K(A)_{\sigma,\tau}=A_{\sigma^{-1}\tau}
\to K(A)_{\tau,\sigma}=A_{\tau^{-1}\sigma}$ makes $K(A)$ into a Hopf $\Vv$-category.
\end{proof}

\section{Hopf categories and weak Hopf algebras}\selabel{5}
Let $A$ be a $k$-linear Hopf category, with $|A|=X$ a finite set, and consider
$$A=\oplus_{x,y\in X} A_{x,y}.$$
We define a multiplication on $A$ in the usual way: for $h\in A_{x,y}$ and $k\in A_{z,u}$,
the product of $hk$ is the image of $h\ot k$ under the map $m_{x,y,u}: A_{x,y}\ot A_{y,u}\to A_{x,u}$ if
$y=z$, and $hk=0$ if $y\neq z$. This multiplication is extended linearly to the whole of $A$.
Then $A$ is a $k$-algebra with unit $1=\sum_{x\in X} 1_x$, where $1_x$ is the identity
morphism $x\to x$.\\
Now we define $\Delta:\ A\to A\ot A$, $\varepsilon:\ A\to k$, $S:\ A\to A$ in such a way that
their restrictions to $A_{x,y}$ are respectively $\Delta_{x,y}$, $\varepsilon_{x,y}$ and
$S_{x,y}$.

\begin{proposition}\prlabel{5.1}
Let $A$ be a $k$-linear Hopf category, with $|A|=X$ a finite set. Then
 $A=\oplus_{x,y\in X} A_{x,y}$ is a weak Hopf algebra.
 \end{proposition}
 
 \begin{proof}
 We refer to \cite{BohmNSI} for the definition of a weak Hopf algebra. We compute that
 $$\Delta(1)=1_{(1)}\ot 1_{(2)}=\sum_{x\in X} 1_x\ot 1_x,$$
 and
 $$1_{(1)}\ot 1_{(2)}1_{(1')}\ot 1_{(2')}=\sum_{x,y\in X} 1_x\ot 1_x1_y\ot 1_y
 =\sum_{x\in X} 1_x\ot 1_x\ot 1_x=(\Delta\ot A)(\Delta(1)),$$
 as needed. In a similar way, we show that 
 $$1_{(1)}\ot 1_{(1')}1_{(2)}\ot 1_{(2')}=(\Delta\ot A)(\Delta(1)).$$
Let us now show that
 $$\varepsilon(hkl)=\varepsilon(hk_{(1)})\varepsilon(k_{(2)}l).$$
 It suffices to show this for $h\in A_{x,y}$, $k\in A_{y',z'}$, $l\in A_{z,u}$. If
 $y\neq y'$ or $z\neq z'$, then both sides of the equation are $0$. 
 Assume that $y= y'$ and $z= z'$. From 
\equref{2.1.3}, it follows that
$\varepsilon(hk_{(1)})\varepsilon(k_{(2)}l)=\varepsilon(h)\varepsilon(k_{(1)})\varepsilon(k_{(2)}l)
 =\varepsilon(h)\varepsilon(\varepsilon(k_{(1)})k_{(2)}l)=
 \varepsilon(h)\varepsilon(kl) =\varepsilon(hkl)$. Similar arguments show that
 $$\varepsilon(hkl)=\varepsilon(hk_{(2)})\varepsilon(k_{(1)}l).$$
 This proves that $A$ is a weak bialgebra. For $h\in A_{x,y}$, we compute that
 $$
 \varepsilon_t(h)=\sum_{z\in X} \lan \varepsilon, 1_zh\ran 1_z
 =\lan\varepsilon, 1_xh\ran 1_x=\lan \varepsilon_{x,y},h\ran 1_x.$$
In a similar way, we show that $\varepsilon_s(h)=\lan\varepsilon,h1_y\ran 1_y=
\lan \varepsilon_{x,y},h\ran 1_y$. Now
\begin{eqnarray*}
h_{(1)}S_{x,y}(h_{(2)})\equal{\equref{2.2.1}} \eta_x(\varepsilon_{x,y}(h))=\varepsilon_t(h);\\
S_{x,y}(h_{(1)})h_{(2)}\equal{\equref{2.2.2}} \eta_y(\varepsilon_{x,y}(h))=\varepsilon_s(h),
\end{eqnarray*}
and, finally,
$$S_{x,y}(h_{(1)})h_{(2)}S_{x,y}(h_{(3)})=\varepsilon_{x,y}(h_{(1)})1_yS_{x,y}(h_{(2)})=
S_{x,y}(h).$$
 \end{proof}
 
 \begin{remark}
 Let $G$ be a groupoid. Using \exref{2.7}, we obtain a $k$-linear Hopf category.
 Then applying \prref{5.1}, we find a weak Hopf algebra, which is precisely the
 groupoid algebra $kG$.
 \end{remark}
 
Now let $C$ be a dual $k$-linear Hopf category. Then every $C_{x,y}$ is an algebra,
 and we have $k$-linear maps $\Delta_{x,y,z}:\ C_{x,z}\to C_{x,y}\ot C_{y,z}$,
 $\varepsilon_x:\ C_{x,x}\to k$ and $S_{x,y}:\ C_{y,x}\to C_{x,y}$ such that the
 following axioms are satisfied, for all $h,k\in C_{x,z}$ and $l,m\in C_{x,x}$:
 \begin{eqnarray}
 \Delta_{x,u,y}(h_{(1,y)})\ot h_{(2,y)}&=&h_{(1,y)}\ot \Delta_{y,u,z}(h_{(2,y)})\eqlabel{5.2.1}\\
 \varepsilon_x(h_{(1,x)})h_{(2,x)}&=&h_{(1,z)}\varepsilon_z(h_{(2,z)})=h;\eqlabel{5.2.2}\\
 \Delta_{x,y,z}(hk)&=&h_{(1,y)}k_{(1,y)}\ot h_{(2,y)}k_{(2,y)};\eqlabel{5.2.3}\\
 \varepsilon_x(lm)&=&\varepsilon_x(l)\varepsilon_x(m);\eqlabel{5.2.4}\\
 \Delta_{x,y,z}(1_{x,z})&=&1_{x,y}\ot 1_{y,z};\eqlabel{5.2.5}\\
 \varepsilon_x(1_{x,x})&=&1;\eqlabel{5.2.6}\\
 l_{(1,y)}S_{x,y}( l_{(2,y)})&=&\varepsilon_x(l)1_{x,y};\eqlabel{5.2.7}\\
S_{y,x}(l_{(1,y)}) l_{(2,y)}&=&\varepsilon_x(l)1_{y,x}.\eqlabel{5.2.8}
\end{eqnarray}
Cere $1_{x,y}$ is the unit element of $C_{x,y}$, and we used the Sweedler-Heyneman notation
$$\Delta_{x,y,z}(h)=h_{(1,y)}\ot h_{(2,y)}.$$

\begin{proposition}\prlabel{5.2}
Let $C$ be a dual $k$-linear Hopf category, and assume that $|C|=X$ is finite.
Then $C=\oplus_{x,y\in X}C_{x,y}$ is a weak Hopf algebra.
\end{proposition}

\begin{proof}
Being the direct product of a finite number of $k$-algebras, $C$ is itself a $k$-algebra,
with unit $1=\sum_{x,z\in X} 1_{x,z}$.
We define a comultiplication on $C$ as follows:
$$\Delta(h)=\sum_{y\in X} \Delta_{x,y,z}(h),$$
for $h\in C_{x,z}$. It follows immediately from \equref{5.2.1} that $\Delta$ is coassociative.
The counit is defined by ($h\in C_{x,y}$):
$$\varepsilon(h)=\begin{cases}
\varepsilon_x(h)&{\rm if~} x=y\\
0&{\rm if~} x\neq y
\end{cases}$$
We verify the left counit condition:
$$
((\varepsilon\ot C)\circ \Delta)(h)=\sum_{y\in X} \varepsilon(h_{(1,y)})h_{(2,y)}=
\varepsilon_x(h_{(1,x)})h_{(2,x)}\equal{\equref{5.2.2}}h.$$
The right counit condition can be verified in a similar way, and we conclude that
$C$ is a coalgebra. It follows from \equref{5.2.3} and \equref{5.2.4} that $\Delta$ 
and $\varepsilon$ preserve the multiplication. it follows from \equref{5.2.5}
that 
$$\Delta(1)=1_{(1)}\ot 1_{(2)}=\sum_{x,y,z\in X} 1_{x,y}\ot 1_{y,z}.$$
We now find easily that
\begin{eqnarray*}
&&\hspace*{-2cm}
1_{(1)}\ot 1_{(2)}1_{(1')}\ot 1_{(2')}=
\sum_{x,y,z,u,v,w\in X} 1_{x,y}\ot 1_{y,z}1_{u,v}\ot 1_{v,w}\\
&=&\sum_{x,y,z,w\in X} 1_{x,y}\ot 1_{y,z}\ot 1_{z,w}=1_{(1)}\ot 1_{(2)}\ot 1_{(3)}.
\end{eqnarray*}
In a similar way, we find that $1_{(1)}\ot 1_{(1')}1_{(2)}\ot 1_{(2')}=
1_{(1)}\ot 1_{(2)}\ot 1_{(3)}$.
Now take $h,k,l\in C_{x,x}$.
\begin{eqnarray*}
&&\hspace*{-2cm}
\varepsilon(hk_{(1)})\varepsilon(k_{(2)}l)=
\sum_{y\in X} \varepsilon(hk_{(1,y)})\varepsilon(k_{(2,y)}l)
= \varepsilon_x(hk_{(1,x)})\varepsilon_x(k_{(2,x)}l)\\
&\equal{\equref{5.2.4}}& \varepsilon_x(h)\varepsilon_x(k_{(1,x)})\varepsilon_x(k_{(2,x)}l)
= \varepsilon_x(h)\varepsilon_x(\varepsilon_x(k_{(1,x)})k_{(2,x)}l)\\
&\equal{\equref{5.2.2}}&\varepsilon_x(h)\varepsilon_x(kl)
\equal{\equref{5.2.4}} \varepsilon_x(hkl)=\varepsilon_x(hkl).
\end{eqnarray*}
We conclude that
\begin{equation}\eqlabel{5.2.9}
\varepsilon(hk_{(1)})\varepsilon(k_{(2)}l)=\varepsilon (hkl),
\end{equation}
if $h,k,l\in C_{x,x}$. If $h,k,l\in C_{x,y}$ with $y\neq x$, then both sides of \equref{5.2.9} are
zero. So we can conclude that \equref{5.2.9} holds for all $h,k,l\in C$. In a similar way, we
can show that
$$\varepsilon(hk_{(2)})\varepsilon(k_{(1)}l)=\varepsilon \Delta_{(g\circ f)(x),(g'\circ f)(x),(g'\circ f')(x)}(hkl),$$
for all $h,k,l\in C$. This shows that $C$ is a weak bialgebra.\\
Recall from \cite{BohmNSI} that the maps $\varepsilon_s,~\varepsilon_t:\ C\to C$
are given by the formulas
$$\varepsilon_s(h)=1_{(1)}\varepsilon(h1_{(2)})~~;~~\varepsilon_t(h)=
\varepsilon(1_{(1)}h)1_{(2)}.$$
These maps can be easily computed: for $h\in C_{x,z}$, we have
$$\varepsilon_t(h)=\sum_{u,v,y\in X} \varepsilon(1_{u,v}h)1_{v,y}
=\sum_{y\in X} \varepsilon(h)1_{z,y}=
\begin{cases}
\sum_{y\in X} \varepsilon_x(h)1_{x,y} &{\rm if~}x= z\\
0&{\rm if~}x\neq z
\end{cases}$$
In a similar way, we find that
$$\varepsilon_s(h)=
\begin{cases}
\sum_{y\in X} \varepsilon_x(h)1_{y,x} &{\rm if~}x= z\\
0&{\rm if~}x\neq z
\end{cases}$$
Now we define $S:\ C\to C$ as follows: the restriction of $S$ to $C_{x,y}$ is
$S_{y,x}$, and then we extend linearly. Then we have, for $h\in C_{x,z}$:
$$(S*C)(h)=\sum_{y\in X} S_{y,x}(h_{(1,y)})h_{(2,y)}.$$
If $x\neq z$, then we find easily that $(S*C)(h)=0=\varepsilon_s(h)$. If $x=z$, then
we find
$$(S*C)(h)\equal{\equref{5.2.8}}\sum_{y\in X} \varepsilon_x(h)1_{y,x}=\varepsilon_s(h).$$
This shows that $S*C=\varepsilon_s$. In a similar way, we have that $C*S=\varepsilon_t$.
Finally we have that
$$(S*C*S)(h)=\sum_{y,u\in X} S_{u,x}(h_{(1,y)(1,u)})h_{(1,y)(2,u)}S_{z,y}(h_{(2,y)}).$$
The terms on the right hand side are products of an element of $C_{u,x}$, an
element of $C_{u,y}$ and an element of $C_{z,y}$. These products are zero if $x\neq y$
of $z\neq u$. Hence we find
\begin{eqnarray*}
&&\hspace*{-2cm}
(S*C*S)(h)=S_{z,x}(h_{(1,x)(1,z)})h_{(1,x)(2,z)}S_{z,x}(h_{(2,x)})\\
&\equal{\equref{5.2.8}}&\varepsilon_x(h_{(1,x)})1_{z,x}S_{z,x}(h_{(2,x)})
\equal{\equref{5.2.2}}S_{x,z}(h)=S(h).
\end{eqnarray*}
This proves that $C$ satisfies all the axioms of a weak Hopf algebra,
see \cite{BohmNSI}.
\end{proof}

\begin{remark}
If $A$ be a $k$-linear Hopf category, with $|A|=X$ an infinite set, then $A=\oplus_{x,y\in X} A_{x,y}$ is an algebra without unit, but with (idempotent) local units. We believe that if $A$ is a Hopf category and using similar constructions as above, the associated algebra $A$ can be endowed with the structure of a {\em weak multiplier Hopf algebra} (see \cite{VanDaeleWang} and \cite{BohmGT}), but we haven't worked out the details of this construction.
\end{remark}

\section{Hopf categories and duoidal categories}\selabel{6}
Let $X$ be a set. We have seen in \seref{1} that $(\Mm_k(X),\bullet,J)$ is a monoidal
category. We will define a second monoidal structure on $\Mm_k(X)$, in such a way that
$\Mm_k(X)$ becomes a duoidal category (also called 2-monoidal category)  in the sense of 
\cite{Aguiar}. We will follow the notation
of \cite{BCZ}, and we call $\bullet$ the black tensor product on $\Mm_k(X)$. The second tensor
product is called the white tensor product and is defined as follows. For $M,N\in \Mm_k(X)$, let
$$(M\odot N)_{x,z}=\oplus_{y\in X} M_{x,y}\ot N_{y,z}.$$
The unit object for the white tensor product is $I$, defined by
$$I_{x,y}=
\begin{cases}
ke_{x,x}&{\rm if}~x=y\\ 0&{\rm if}~x\neq y
\end{cases}$$
We will simply write
$$I_{x,y}=k\delta_{x,y},$$
where the Kronecker symbol $\delta_{x,y}$ stands formally for the element of the identity matrix
in the $(x,y)$-position. Let
$$\tau:\ I\to J$$
be the natural inclusion. We compute that
$$(I\bullet I)_{x,y}=k\delta_{x,y}\ot k\delta_{x,y}=k\delta_{x,y}=I_{x,y},$$
hence $I\bullet I=I$, and we let
$$\delta:\ I\to I\bullet I$$
be the identity map. Now we compute that
$$(J\odot J)_{x,y}=\oplus_{z\in X} ke_{x,z}\ot ke_{z,y}=
\oplus_{z\in X} kze_{x,y}=kXe_{x,y}.$$
We now define $\varpi:\ J\odot J\to J$. For all $x,y\in X$,
$$\varpi_{x,y}:\ \oplus_{z\in X} kze_{x,y}\to k e_{x,y},~~\varpi_{x,y}(\sum_{z\in X} \alpha_z ze_{x,y})=
\sum_{z\in X} \alpha_z e_{x,y}.$$
For $M,N,P,Q\in \Vv(X)$ we have that
\begin{eqnarray*}
&&\hspace*{-2cm}
((M\bullet N)\odot (P\bullet Q))_{x,y}=\bigoplus_{z\in X} M_{x,z}\ot N_{x,z}\ot P_{z,y}\ot Q_{z,y};\\
&&\hspace*{-2cm}
((M\odot P)\bullet (N\odot Q))_{x,y}=\bigoplus_{u,v\in X} M_{x,u}\ot P_{u,y}\ot N_{x,v}\ot  Q_{v,y},
\end{eqnarray*}
and we define
$$\zeta_{M,N,P,Q}:\ (M\bullet N)\odot (P\bullet Q)\to (M\odot P)\bullet (N\odot Q)$$
as follows: for $x,y\in X$, $\zeta_{M,N,P,Q,x,y}$ is the map switching the second and third tensor
factor, followed by the natural inclusion.

\begin{theorem}\thlabel{6.1}
Let $X$ be a set.
$(\Mm_k(X), \odot, I, \bullet, J, \delta, \varpi,\tau, \zeta)$ is a duoidal category.
\end{theorem}

\begin{proof}
We have to show that the axioms in \cite[Def. 1.1]{BCZ} are satisfied.\\
1) $(J,\varpi,\tau)$ is a monoid in $(\Mm_k(X),\odot, I)$.\\
Associativity: first compute that
$$(J\odot J\odot J)_{x,y}=k(X\times X) e_{x,y}=\oplus_{u,v\in X} k(u,v)e_{x,y},$$
and
\begin{eqnarray*}
&&\hspace*{-2cm}
(\varpi(J\odot \varpi))(\sum_{u,v} \alpha_{(u,v)}(u,v) e_{x,y})
= \varpi(\sum_{u,v} \alpha_{(u,v)}u e_{x,y})\\
&=& \sum_{u,v} \alpha_{(u,v)} e_{x,y}=( \varpi( \varpi\odot J))(\sum_{u,v} \alpha_{(u,v)}(u,v) e_{x,y}).
\end{eqnarray*}
Left unit property: we have to show that the diagram
$$\xymatrix{(J\odot I)_{x,y}\ar[drr]_{=}\ar[rr]^{(J\odot \tau)_{x,y}}&&
(J\odot J)_{x,y}\ar[d]^{\varpi_{x,y}}\\
&&J_{x,y}}$$
commutes, for all $x,y\in X$. Observe that
$(J\odot I)_{x,y}=\oplus_{z\in X} ke_{x,z}\ot k\delta_{z,y}=ke_{x,y}=J_{x,y}$ and
$(J\odot J)_{x,y}=kXe_{x,y}$. Now
$$\varpi_{x,y}\bigl((J\odot \tau)_{x,y}(\alpha e_{x,y}\bigr)=
\varpi_{x,y}(\alpha y e_{x,y})= \alpha e_{x,y},$$
for all $\alpha\in k$. The right unit property can be shown in a similar way.\\
2) $(I,\delta,\tau)$ is a comonoid in $(\Mm_k(X),\bullet,J)$.\\
The coassociativity of $\delta$ is clear, since $\delta$ is the identity map. For the left
counit property: oberve that the diagram
$$\xymatrix{I_{x,y}=k\delta_{x,y}\ar[d]_{\delta_{x,y}}\ar[drr]^{=}\\
I_{x,y}=k\delta_{x,y}\ar[rr]^(.4){(J\bullet \tau)_{x,y}}&&(J\bullet I)_{x,y}=k\delta_{x,y}}$$
commutes: the three maps in the diagram are the identity map.\\
3) Verification of the associativity and unitality axioms \cite[1.6-7]{BCZ} is obvious
and is left to the reader.
\end{proof}

Recall the following definition from \cite[Def. 6.25]{Aguiar} (see also \cite[Def. 1.2]{BCZ}).

\begin{definition}\delabel{6.2}
Let $(\Mm,\odot, I, \bullet, J, \delta, \varpi,\tau, \zeta)$ be a duoidal category.
A bimonoid is an object $A$, together with an algebra structure $(\mu,\eta)$
in $(\Mm,\odot, I)$
and a coalgebra structure $(\Delta,\varepsilon)$ in $(\Mm,\bullet, J)$ subject to
the compatibility conditions
\begin{eqnarray}
\Delta\circ \mu&=&(\mu\bullet\mu)\circ \zeta\circ (\Delta\odot\Delta);\eqlabel{6.2.1}\\
\varpi\circ (\varepsilon\odot\varepsilon)&=&\varepsilon\circ\mu;\eqlabel{6.2.2}\\
(\eta\bullet\eta)\circ \delta&=&\Delta\circ\eta;\eqlabel{6.2.3}\\
\varepsilon\circ\eta&=&\tau.\eqlabel{6.2.4}
\end{eqnarray}
\end{definition}

\begin{theorem}\thlabel{6.3}
Let $X$ be a set, and let $A\in \Mm_k(X)$. We have a bijective correspondence between bimonoid
structures on $A$ over the
duoidal category $(\Mm_k(X), \odot, I, \bullet, J, \delta, \varpi,\tau, \zeta)$ from
\thref{6.1} and $\Cog(\Mm_k)$-category structures on $A$.
\end{theorem}

\begin{proof}
First let $A$ be a bimonoid. $A$ has an algebra structure $(\mu,\eta)$ on $(\Mm_k(X),\odot,I)$.
Consider the $(x,y)$-component of the multiplication map $\mu:\ A\odot A\to A$, namely
$$\mu_{x,y}:\ \oplus_{u\in X} A_{x,u}\ot A_{u,y}\to A_{x,y},$$
and let $\mu_{x,z,y}$ be the composition
$$\mu_{x,y}\circ i_z:\ A_{x,z}\ot A_{z,y}\to \oplus_{u\in X} A_{x,u}\ot A_{u,y}\to A_{x,y},$$
where $i_z$ is the natural inclusion. Also consider the $(x,x)$-component of the unit map
$\eta:\ I\to A$, namely $\eta_x=\eta_{x,x}:\ k\to A_{x,x}$. Now it is easy to see that
(\ref{eq:1.1}-\ref{eq:1.2}) are satisfied, so that $A$ becomes a $k$-linear category.\\
$A$ has a coalgebra structure $(\Delta,\varepsilon)$ on $(\Mm_k(X),\bullet,J)$. Consider
the $(x,y)$-component of the comultiplication $\Delta:\ A\to A\bullet A$ and of the counit
$\varepsilon:\ A\to J$. This gives $k$-linear maps $\Delta_{x,y}:\ A_{x,y}\to A_{x,y}\ot A_{x,y}$
and $\varepsilon_{x,y}:\ A_{x,y}\to k$ making $A_{x,y}$ into a $k$-coalgebra.\\
Now we write the $(x,y)$-component of \equref{6.2.1} and \equref{6.2.2} as commutative diagrams. This gives us
$$\xymatrix{
\oplus_z A_{x,z}\ot A_{z,y}
\ar[d]_{\oplus_z \Delta_{x,z}\ot \Delta_{z,y}}
\ar[r]^(.6){\mu_{x,y}}&
A_{x,y}\ar[r]^(.4){\Delta_{x,y}}&
A_{x,y}\ot A_{x,y}\\
\oplus_z A_{x,z}\ot A_{x,z}\ot A_{z,y}\ot A_{z,y}
\ar[rr]^{\zeta_{x,y}}&&
\oplus_{u,v} A_{x,u}\ot A_{u,y}\ot A_{x,v}\ot A_{v,y}
\ar[u]_{\mu_{x,y}\ot \mu_{x,y}}
}$$
and
$$\xymatrix{
\oplus_z A_{x,z}\ot A_{z,y}
\ar[d]_{\mu_{x,y}}\ar[rr]^(.4){\oplus_z \varepsilon_{x,z}\ot \varepsilon_{z,y}}&&
\oplus_z ke_{x,z}\ot e_{z,y}=\oplus_z kze_{x,y}\ar[d]^{\varpi}\\
A_{x,y}\ar[rr]^{\varepsilon_{x,y}}&&ke_{x,y}
}$$
Evaluating the two diagrams at $a\ot b\in A_{x,z}\ot A_{z,y}$, we find that
$$\Delta_{x,y}(ab)=a_{(1)}b_{(1)}\ot a_{(2)}b_{(2)}~~{\rm and}~~
\varepsilon_{x,y}(ab)=\varepsilon_{x,z}(a)\varepsilon_{y,z}(b).$$
Now we write the $(x,x)$-component of \equref{6.2.3} and \equref{6.2.4} as commutative diagrams. This gives
$$\xymatrix{
k\ar[d]_{\delta_{x,x}}\ar[r]^{\eta_x}&
A_{x,x}\ar[d]^{\Delta_{x,x}}\\
k\ot k \ar[r]^(.4){\eta_x\ot \eta_x}& A_{x,x}\ot A_{x,x}
}~~~{\rm and}~~~
\xymatrix{
k\ar[r]^(.4){\eta_x}\ar[rd]_{\tau_{x,x}}&
A_{x,x}\ar[d]^{\varepsilon_{x,x}}\\
&k}$$
Evaluating these diagrams at $1$, we find that $\Delta_{x,x}(1_x)=1_x\ot 1_x$ and $\varepsilon_{x,x}(1_x)=1$,
and we conclude that $A$ is a $\Cog(\Mm_k)$-category.\\
Conversely, let $A$ be a $\Cog(\Mm_k)$-category. Define $\mu:\ A\odot A\to A$, $\eta:\ I\to A$, $\Delta:\ A\to A\bullet A$
and $\varepsilon:\ A\to J$ as follows. $\mu_{x,y}=\sum_u \mu_{x,u,y}:\
\oplus_{u\in X} A_{x,u}\ot A_{u,y}\to A_{x,y}$; $\eta_{x,y}=0$ if $x\neq y$ and $\eta_{x,x}=\eta_x$; the components of
$\Delta$ and $\varepsilon$ are just $\Delta_{x,y}$ and $\varepsilon_{x,y}$. Straightforward computations show that
this turns $A$ into a bimonoid. It is clear that these two operations are inverses. This completes the proof.
\end{proof}

\subsection*{Linearization and the duoidal category of spans}
We have seen in \thref{6.1} that we can associate a duoidal category $\Mm_k(X)$ to a set $X$.
In \cite{Aguiar,BCZ}, two other classes of duoidal categories are investigated, namely
the category ${\rm span}(X)$ consisting of spans, and the category ${}_R\Mm_R$
of bimodules over a commutative $k$-algebra $R$. We will now discuss how these three
classes of examples are related. To this end, we need to give alternative descriptions of
$\Mm_k(X)$ and ${\rm span}(X)$.\\
As we have seen in \exref{2.11}, every set $X$ carries a unique comonoid structure in $\Sets$. 
A right $X$-coaction on a set $V$ consists of a map $\rho:\ V\to V\times X$ of the form $\rho(v)=(v,s(v))$,
where $s:\ V\to X$ is a function. So right $X$-coactions on $V$ correspond bijectively to $X^V$.
In a similar way, giving a two-sided coaction of $X$ on $V$ amounts to giving two functions $s,t:\ V\to X$,
which means precisely that $(V,t,s)$ is a span, see \cite[Sec. 4.2]{BCZ}. Morphisms of spans correspond
to bicomodule maps, and we conclude that the categories ${}^X\Sets^X$ and ${\rm span}(X)$ are isomorphic.
The white product of two spans $V$ and $W$ is
$$V\odot W=\{(v,w)\in V\times W~|~s(v)=t(w)\}$$
is precisely the cocarthesian product $V\times^X W$. Now observe that the category ${}^X\Sets^X$
is isomorphic to $\Sets^{X\times X}$.  The black product is
$$V\bullet W=\{(v,w)\in V\times W~|~s(v)=s(w),~t(v)=t(w)\}$$
and this is the cocarthesian product $V\times^{X\times X} W$. The white unit object is $X$, and the black unit object
is $X\times X$.\\
A similar description applies to $\Mm_k(X)$. $kX$ is a coalgebra, and we have isomorphisms of categories
$$\Mm_k(X)\cong {}^{kX}\Mm^{kX}_k\cong \Mm_k^{k(X\times X)}.$$
An object $(M_{x,y})_{x,y\in X}$ corresponds to $M=\oplus_{x,y\in X} M_{x,y}$, with left and right $kX$-coaction
given by the formulas
$$\lambda(m)=x\ot m~~;~~\rho(m)=m\ot y,$$
for $m\in M_{x,y}$, extended linearly. The black tensor product in $\Mm_k(X)$ is precisely the cotensor product over
$k(X\times X)$, and the white one is the cotensor product over $kX$.\\
The linearization functor $L:\ \Sets\to \Mm_k$ is strongly monoidal, sends $X$ to the grouplike coalgebra $kX$
and a set $V$ with a two-sided $X$-coaction to the $kX$-bicomodule $kV$. We find the following result.

\begin{proposition}\prlabel{6.4}
The linearization functor induces a functor $L:\ {\rm span}(X)\to \Mm_k(X)$ preserving the black and white
tensor products.
\end{proposition}

This construction can be generalized, replacing $kX$ by a cocommutative coalgebra $C$. We have to assume
that the cotensor product is associative, which can be done by requiring that $k$ is a field, or else that $k$ is
a commutative ring and that $C$ is finitely generated and projective over $k$. Then the category
${}^C\Mm_k^C\cong \Mm_k^{C\ot C}$ of $C$-bicomodules is duoidal, with the cotensor product over $C$ and $C\ot C$ as the
white and black tensor product. This brings us back to the second example of duoidal category studied in \cite{Aguiar,BCZ}.
For a commutative $k$-algebra $A$, the category ${}_A\Mm_A\cong \Mm_{A\ot A}$ is a duoidal category,
with the tensor products over $A$ and $A\ot A$ as the black and white tensor product. This is precisely the dual construction.

\subsection*{Generalized Hopf monoids in monoidal bicategories}
Now we focus attention to the recent work by B\"ohm and Lack \cite{BL} on generalized Hopf monoids in monoidal bicategories.\\
It is well-known that the category of endomorphisms of an object of a bicategory is a monoidal category. 
It was observed in \cite{Street} that, in a similar way, duoidal categories arise as the category of endomorphisms
in a monoidal bicategory of a pseudomonoid whose multiplication $1$-cell and unit $1$-cell have a right adjoint (such an object is known as a {\em map-monoidale}). In this case, the second monoidal structure is obtained using a {\em convolution product}.
Consider the monoidal bicategory of free $k$-coalgebras, bicomodules and bicomodule maps, with the cotensor product as horizontal composition, the opposite composition as vertical composition and the $k$-tensor product as monoidal product.
$kX$ is a map-monoidale in this monoidal bicategory. Hence the category $\Mm_k(X)\cong {}^{kX}\Mm_k^{kX}$ of $kX$-bicomodules
is the category of endomorphisms over a map-monoidale, so it can be endowed with a duoidal structure.
This duoidal structure coincides with the one described above, the black monoidal product being the convolution product. It also follows from \cite{Street} that $A$ is a bimonoid over the duoidal endohom category $\Mm_k(X)$ if and only if it is a monoidal comonad on $kX$ in the monoidal bicategory described above, hence it induces a monoidal comonad on $\Mm_k(X)$. 

Furthermore, B\"ohm and Lack provide equivalent conditions for the bimonoid $A$ in the duoidal endohom category to have an antipode (i.e. to be a Hopf monoid), in terms of a fundamental theorem of Hopf modules (see also our \seref{9}) and in terms of the associated monoidal comonad to be a Hopf (co)monad. 
In particular, this leads us to the following result. 

\begin{theorem}\thlabel{6.5}
Let $X$ be a set, and let $A\in \Mm_k(X)$. We have a bijective correspondence between Hopf monoid
structures on $A$ (in the sense of \cite{BL}) over the
duoidal category $(\Mm_k(X), \odot, I, \bullet, J, \delta, \varpi,\tau, \zeta)$ from
\thref{6.1} and Hopf $\Mm_k$-category structures on $A$. In particular, if $A$ is Hopf $\Mm_k$-category, then this induces a Hopf monad on $\Mm_k(X)$. 
\end{theorem}

\begin{proof}
From the discussion above, we already know that the structure of an $\Cog(\Mm_k)$-category on $A$ corresponds to the structure of a bimonoid in the duoidal category $\Mm_k(X)$. Hence it only remains to compare the antipode axioms for Hopf categories \equref{2.2.1} and \equref{2.2.2} with the antipode axioms of \cite[Theorem 7.2]{BL}. We leave out the details, but remark that the monoidal bicategory of bicomodules over free coalgebras has duals. Given a $kX$-$kY$ bicomodule $M=\oplus_{(x,y)\in X\times Y}M_{x,y}$, then $M^{-}=M^{op}=\oplus_{(y,x)\in Y\times X}M_{y,x}$ is a $kY$-$kX$ bicomodule. Furthermore, the $2$-cell $\varphi$ in \cite{BL} should in our setting be interpreted as the inclusion map $A_{x,y}\ot A_{y,x}\to \oplus_{y\in X} A_{x,y}\ot A_{y,x}$.
\end{proof}

\section{Hopf categories and Morita contexts}\selabel{6b}
Let $k$ be a commutative ring, and $\Vv=\Mm_k$, the category of $k$-modules.

\begin{definition}\delabel{6b.1}
A Morita context consists of the following data:
\begin{enumerate}
\item a class $X$;
\item $A_{x,x}$ is a $k$-algebra, for all $x\in X$;
\item $A_{x,y}$ is an $(A_{x,x},A_{y,y})$-bimodule, for all $x,y\in X$;
\item $\ol{m}_{x,y,z}:\ A_{x,y}\ot_{A_{y,y}}A_{y,z}\to A_{x,z}$ is an 
$(A_{x,x},A_{z,z})$-bimodule map,
\end{enumerate}
satsifying the following conditions:
\begin{enumerate}
\item $\ol{m}_{x,x,y}:\ A_{x,x}\ot_{A_{x,x}}A_{x,y}\to A_{x,y}$
and
$\ol{m}_{x,y,y}:\ A_{x,y}\ot_{A_{y,y}}A_{y,y}\to A_{x,y}$
are the canonical isomorphisms;
\item the associativity condition \equref{6b1.1} is satisfied, for all $x,y,z,u\in X$
\end{enumerate}
\begin{equation}\eqlabel{6b1.1}
\ol{m}_{x,y,u}\circ (A_{x,y}\ot_{A_{y,y}} \ol{m}_{y,z,u})=
\ol{m}_{x,z,u}\circ (\ol{m}_{x,y,z}\ot_{A_{z,z}}A_{z,u}).
\end{equation}
For $a\in A_{x,y}$ and $b\in A_{y,z}$, we will write
$\ol{m}_{x,y,z}(a\ot_{A_{y,y}}n)=ab$.
\end{definition}

Morita contexts can be organized into a 2-category ${}_k\dul{\rm Mor}$. 
Before we describe the 1-cells, we recall the following result. Let $f:\
A\to B$ be a morphism of $k$-algebras, and consider $M,N\in \Mm_A$,
$M',N'\in \Mm_B$, and $k$-linear maps $g:\ M\to M'$ and $h:\ N\to N'$ such
that $g(ma)=g(m)f(a)$ and $h(an)=f(a)h(n)$, for all $a\in A$, $m\in M$
and $n\in N$. Then we have a well-defined map
$$g\ot_f h:\ M\ot_A N\to M'\ot_BN',~~(g\ot_f h)(m\ot_A n)=g(m)\ot_B h(n).$$
A 1-cell $f:\ A\to B$ in ${}_k\dul{\rm Mor}$ consists of $f:\ X\to Y$, and
maps $f_{x,y}:\ A_{x,y}\to B_{f(x),f(y)}$ such that
\begin{itemize}
\item every $f_{x,x}$ is an algebra map;
\item $f_{x,y}(a'aa'')=f_{x,x}(a')f_{x,y}(a)f_{y,y}(a'')$, for all $a'in A_{x,x}$,
$a\in A_{x,y}$ and $a''\in A_{y,y}$;
\item $f_{x,y}\circ \ol{m}_{x,y,z}=\ol{m}_{f(x),f(y),f(z)}\circ (f_{x,y}\ot_{f_{y,y}}f_{y,z})$.
\end{itemize}
For two given 1-cells $f,g:\ A\to B$, a 2-cell $\alpha:\ f\Rightarrow g$ consists of
a family of elements $\alpha_x\in B_{g(x),f(x)}$ indexed by $x$ such that
$$\ol{m}_{g(x),g(y),f(y)}(g_{x,y}(a)\ot_{B_{g(y),g(y)}}\alpha_y)=
\ol{m}_{g(x),f(x),f(y)}(\alpha_x\ot_{B_{g(x),g(x)}}f_{x,y}(a)),$$
for all $x,y\in X$ and $a\in A_{x,y}$.\\

Let $A$ be a Morita context, and take $x\neq y\in X$. Take $p,r\in A_{x,y}$ and $q\in A_{y,x}$.
It follows from \equref{6b1.1} that
$$\ol{m}_{x,y,x}(p\ot_{A_{y,y}}q)r=p\ol{m}_{y,x,y}(q\ot_{A_{x,x}} r).$$
It follows that $(A_{x,x},A_{y,y},A_{x,y},A_{y,x},\ol{m}_{x,y,x},\ol{m}_{y,x,y})$ is a Morita
context. In particular, Morita contexts with a pair as underlying class are Morita contexts in the classical sense.

\begin{theorem}\thlabel{6b.2}
The 2-categories ${}_{\Mm_k}\Cat$ and
${}_k\dul{\rm Mor}$ are isomorphic.
\end{theorem}

\begin{proof} (sketch)
Let $A$ be a $k$-linear category, with underlying class $X$. It is clear that $A_{x,x}$ is a $k$-algebra,
and that $A_{x,y}$ is an $(A_{x,x},A_{y,y})$-bimodule, for all $x,y\in X$. Take $a\in A_{x,y}$, $b\in A_{y,y}$ and
$c\in A_{y,z}$. From \equref{1.1}, it follows that
$m_{x,y,z}(ab\ot c)=m_{x,y,z}(a\ot bc)$, so we have a well-defined map
$$\ol{m}_{x,y,z}:\ A_{x,y}\ot_{A_{y,y}} A_{y,z}\to A_{x,z},~~\ol{m}_{x,y,z}(a\ot_{A_{y,y}} c)={m}_{x,y,z}(a\ot c).$$
From \equref{1.2}, it follows that $\ol{m}_{y,y,z}(1_y \ot_{A_{y,y}} c)={m}_{y,y,z}(1_y\ot c)=c$, so that
$\ol{m}_{y,y,z}$ is the canonical isomorphism $A_{y,y}\ot_{A_{y,y}} A_{y,z}\cong A_{y,z}$. It is easy to verify
that the associativity axiom \equref{6b1.1} is satisfied, and it follows that $A$ is a Morita $X$-context.\\
Conversely, let $A$ be a Morita context with underlying class $X$. Define $m_{x,y,z}$ as the composition of $\ol{m}_{x,y,z}$
and the canonical surjection $A_{x,y}\ot A_{y,z}\to A_{x,y}\ot_{A_{y,y}} A_{y,z}$. It is a straightforward
verification to check that $A$ is $k$-linear category.\\
It is clear that these two constructions are inverses, and this defines 2-functors between our two 2-categories
at the level of 0-cells. We leave it to the reader that we have a one-to-one
correspondence between  1-cells and 2-cells in ${}_{\Mm_k}\Cat$ and ${}_k\dul{\rm Mor}$.
\end{proof}

\begin{theorem}\thlabel{6b.3}
Let $A$ be a $k$-linear category with underlying class $X$, and consider the corresponding Morita context. The following
statements are equivalent.
\begin{enumerate}
\item $m_{x,y,z}$ is surjective, for all $x,y,z\in X$
\item $m_{x,y,x}$ is surjective, for all $x,y\in X$;
\item $\ol{m}_{x,y,x}$ is bijective, for all $x,y\in X$;
\item $\ol{m}_{x,y,z}$ is bijective, for all $x,y,z\in X$.
\end{enumerate}
$A$ is called strict if these four equivalent conditions are satisfied.
\end{theorem}

\begin{proof}
The implications
$\ul{4)\Rightarrow 1)\Rightarrow 2)}$ are obvious.\\
$\ul{2)\Rightarrow 3)}$. If $m_{x,y,x}$ is surjective, then $\ol{m}_{x,y,x}$ is
also surjective. We have seen that $(A_{x,x},A_{y,y},A_{x,y},A_{y,x},\ol{m}_{x,y,x},\ol{m}_{y,x,y})$
is a Morita context, hence surjectivity of $\ol{m}_{x,y,x}$ implies injectivity, by a
classical property of Morita contexts, see \cite{Bass}.\\
$\ul{3)\Rightarrow 4)}$. For all $x,y\in X$, we have that $\ol{m}_{x,x,y}$ and $\ol{m}_{x,y,y}$ are
bijective (by definition), and $\ol{m}_{x,y,x}$ is bijective by assumption.
It follows from \equref{6b1.1} that
$$\ol{m}_{x,y,z}\circ (A_{x,y}\ot_{A_{y,y}} \ol{m}_{y,x,z})=
\ol{m}_{x,x,z}\circ (\ol{m}_{x,y,x}\ot_{A_{x,x}}A_{x,z}).$$
The right hand side is invertible, and therefore  $\ol{m}_{x,y,z}\circ (A_{x,y}\ot_{A_{y,y}} \ol{m}_{y,x,z})$
is also invertible. This implies that $\ol{m}_{x,y,z}$ has a right inverse, and
that $A_{x,y}\ot_{A_{y,y}} \ol{m}_{y,x,z}$ has a left inverse. Having a right inverse,
$\ol{m}_{x,y,z}$ is surjective, for all $x,y,z\in X$.\\
It also follows that $A_{y,x}\ot_{A_{x,x}} A_{x,y}\ot_{A_{y,y}} \ol{m}_{y,x,z}$ and
$\ol{m}_{y,x,y} \ot_{A_{y,y}} \ol{m}_{y,x,z}$ have a left inverse, because
$\ol{m}_{y,x,y}$ is bijective. Let $f$ be the left inverse of $\ol{m}_{y,x,y} \ot_{A_{y,y}} \ol{m}_{y,x,z}$,
and take $\alpha\in \Ker \ol{m}_{y,x,z}$. $\ol{m}_{y,x,y}$ is surjective, hence there exists
$\beta\in A_{y,x}\ot_{A_{x,x}} A_{x,y}$ such that $\ol{m}_{y,x,y}(\beta)=1_y$. Now
$$\beta\ot_{A_{y,y}}\alpha=
\bigl( f\circ (\ol{m}_{y,x,y} \ot_{A_{y,y}} \ol{m}_{y,x,z})\bigr)(\beta\ot_{A_{y,y}}\alpha)=0,$$
and
$$0=\ol{m}_{y,x,y}(\beta) \ot_{A_{y,y}}\alpha=1_y\ot_{A_{y,y}}\alpha$$
in $A_{y,y}\ot_{A_{y,y}}A_{y,x}\ot_{A_{x,x}} A_{x,z}\cong A_{y,x}\ot_{A_{x,x}} A_{x,z}$, and, finally,
$\alpha=0$. We conclude that $\ol{m}_{y,x,z}$ is injective.
\end{proof}

\begin{example}\exlabel{6b.4}
The category $A$ of $k$-progenerators, is a strict
$k$-linear category. For two finitely generated projective $k$-modules $P$ and $Q$,
we have that $A_{P,Q}=\Hom(Q,P)$, and $m_{P,Q,P}:\ A_{P,Q}\ot A_{Q,P}\to A_{P,P}$ is given by
composition: $m_{P,Q,P}(f\ot g)=f\circ g$. We have to show that $m_{P,Q,P}$ is surjective.\\
$Q$ is a generator of ${}_k\Mm$, so there exist $q_i\in Q$ and $q_i^*\in Q^*$ such that
$\sum_i \lan q_i^*,q_i\ran =1$.\\
$P$ is finitely generated projective, so there exist $p_j\in P$ and $p_j^*\in P^*$ such that
$p=\sum_j \lan p_j^*,p\ran p_j$, for all $p\in P$. Now consider
\begin{eqnarray*}
f_{ij}:\ Q\to P&;&f_{ij}(q)=\lan q_i^*,q\ran p_j;\\
g_{ij}:\ P\to Q&;&g_{ij}(p)=\lan p_j^*,p\ran q_i.
\end{eqnarray*}
Now
$$m_{P,Q,P}(\sum_{i,j} f_{ij}\ot g_{ij})(p)=\sum_{i,j} \lan p_j^*,p\ran\lan q_i^*,q_i\ran p_j=p,$$
hence $m_{P,Q,P}(\sum_{i,j} f_{ij}\ot g_{ij})=P$ and $m_{P,Q,P}$ is surjective.
\end{example}

\begin{example}\exlabel{6b.4a}
Let $A$ be a $G$-graded $k$-algebra, and consider the corresponding $k$-linear category
$K(A)$ (see \prref{4.2}). $K(A)$ is strict if and only if the multiplication maps
$A_{g^{-1}h}\ot A_{h^{-1}g}\to A_e$ are surjective, for all $g,h\in G$. This is equivalent
to surjectivity of $A_{g^{-1}}\ot A_{g}\to A_e$, for all $g\in G$. This is one of the equivalent
definitions of a strongly graded $k$-algebra, see for example \cite{NVO}. We conclude that $K(A)$
is strict if and only if $A$ is a strongly graded $k$-algebra.
\end{example}

Now assume that $A$ is a $\Cog(\Mm_k)$-category. It follows from the axioms that every
$A_{x,x}$ is a bialgebra and that every
$A_{x,y}$ is an $(A_{x,x},A_{y,y})$-bimodule coalgebra. In this case the induction
functors $A_{x,y}\ot -:\ {}_{A_{y,y}}\Mm\to {}_{A_{x,x}}\Mm$ are comonoidal. 

\begin{example}
Let $H$ be Hopf algebra with bijective antipode $S$, and let 
$A$ be a faithfully flat right $H$-Galois object.
In \cite{Schauenburg}, a new Hopf algebra $L$ is constructed in such a way that
$A$ is a faitfhully flat left $L$-Galois object, and even an $(L,H)$-bigalois
object. $A^{\rm op}$ is an $(H,L)$-bigalois object (see \cite[Remark 4.4]{Schauenburg}).
The left $H$-coaction on $A^{\rm op}$ is the following:
$$\lambda(a)=S^{-1}(a_{[1]})\ot a_{[0]}.$$
We now have a dual $\Alg(\Mm_k)$-category $A$ with underlying class $\{x,y\}$ defined as follows:
$$A_{x,x}= H;~~A_{y,y}=L;~~A_{x,y}=A;~~A_{y,x}=A^{\rm op}.$$
$A$ is even a dual Hopf category; the antipode maps are the following:
$S_H:\ H\to H$, $S_L:\ L\to L$ and the identity $A_{x,y}=A\to A_{y,x}=A^{\rm op}$.\\
Now let $H$ be finitely generated and projective; then $A$ and $L$ are also finitely
generated and projective, and the dual category of $A$ is an example of a $k$-linear
Hopf category.
\end{example}

\section{Hopf modules and the fundamental theorem}\selabel{9}
Let $\Vv$ be a strict monoidal category with equalizers, and let $A$ be a $\Cog(\Vv)$-category,
with underlying class $|A|=X$. Assume that $M\in \Vv(X)$, with the following additional structure:
\begin{itemize}
\item $M\in \Vv_A$ in the sense of \deref{3.1}, with structure morphisms $\psi_{x,y,z}:\ M_{x,y}\ot A_{y,z}\to
M_{x,z}$;
\item $M\in \Vv^A$, that is, $M$ is a right comodule over $A$ considered as a coalgebra in $\Vv(X)$; this
means that every $M_{x,y}$ is a right $A_{x,y}$-comodule, with coaction $\rho_{x,y}:\ M_{x,y}\to M_{x,y}\ot A_{x,y}$.
\end{itemize}
Recall that $A\bullet A$ is also a $\Vv$-category. $M\bullet A\in \Vv_{A\bullet A}$, with structure maps
$$\psi^{M\bullet A}_{x,y,z}=(\psi_{x,y,z}\ot m_{x,y,z})\circ (M_{x,y}\ot c_{A_{x,y},A_{y,z}}\ot A_{y,z}).$$
$M$ is called a Hopf module if the compatibility relation
\begin{equation}\eqlabel{9.0.1}
\rho_{x,z}\circ \psi_{x,y,z}= \psi^{M\bullet A}_{x,y,z} \circ (\rho_{x,y}\ot A_{y,z})
\end{equation}
holds for all $x,y,z\in X$. A morphism between Hopf modules is a morphism in $\Vv$ that is a morphism in
$\Vv_A$ and $\Vv^A$. The category of Hopf modules is denoted $\Vv(X)_A^A$.\\
We introduce the category $\Dd(X)$ ($\Dd$ stands for ``diagonal''). Its objects are families of objects
in $\Vv$ indexed by $X$, and a morphism $N\to N'$ consists of a family of morphisms $N_x\to N'_x$ in $\Vv$.

\begin{proposition}\prlabel{9.1}
We have a pair of adjoint functors $(F,G)$ between $\Dd(X)$ and $\Vv(X)_A^A$.
\end{proposition}

\begin{proof}
We define a functor $F:\ \Dd(X)\to \Vv(X)_A^A$ as follows. For $N\in \Dd(X)$, let $F(N)\in\Vv(X)_A^A$ be given by the data
$$F(N)_{x,y}= N_x\ot A_{x,y};~~\psi_{x,y,z}=N_x\ot m_{x,y,z};~~\rho_{x,y}=N_x\ot \Delta_{x,y}.$$
For $f:\ N\to N'$ in $\Dd(X)$, let $F(f)_{x,y}=f_{x}\ot A_{x,y}$. Verification of further details is straightforward.\\
Now we define $G:\ \Vv(X)_A^A\to \Dd(X)$. Let $M\in \Vv(X)_A^A$. $M_{x,x}$ is a right $A_{x,x}$-module, for every $x\in X$,
and we define $G(M)=M^{{\rm co}A}$ as follows:
$$G(M)_x=M^{{\rm co}A}_x=M_{x,x}^{{\rm co} A_{x,x}},$$
the equalizer of the parallel morphisms $\rho_{x,x},~M_{x,x}\ot \eta_x:\ M_{x,x}\to M_{x,x}\ot A_{x,x}$.
For $g:\ M\to M'$ in $\Vv(X)_A^A$, $G(g)=g^{{\rm co}A}$ is defined as follows: $G(g)_x=g^{{\rm co}A}_x$
is the unique morphism in $\Vv$ making the diagram
$$\xymatrix{M^{{\rm co}A}_x\ar[rr]^{i_x}\ar@{.>}[d]_{\exists ! g^{{\rm co}A}_x}&&
M_{xx}\ar[d]_{f_{x,x}}\ar@<-.5ex>[rr]_(.42){\rho_{x,x}} 
\ar@<.5ex>[rr]^(.43){M_{x,x}\ot \eta_X}&&
M_{x,x}\ot A_{x,x}\ar[d]^{f_{x,x}\ot A_{x.x}}\\
{M'}^{{\rm co}A}_x\ar[rr]^{i_x}&&
M'_{x,x}\ar@<-.5ex>[rr]_(.42){\rho'_{x,x}} 
\ar@<.5ex>[rr]^(.43){M'_{x,x}\ot \eta_X}&&M'_{x,x}\ot A_{x,x}}$$
commutative. The existence and uniqueness of $g^{{\rm co}A}_x$ is guaranteed by the universal property of equalizers.\\
Next we describe the unit and the counit of the adjunction. For $N\in \Dd(X)$, the unit $\eta_N:\ N\ot GF(N)$
has $X$ component $\eta^N_x:\ N_x\to GF(N)_x=(N_x\ot A_{x,x})^{{\rm co}A_{x,x}}$, the unique morphism in $\Vv$
such that
\begin{equation}
i\circ \eta_x^N=N_x\ot \eta_x:\ N_x\to (N_x\ot A_{x,x})^{{\rm co}A_{x,x}}\to N_x\ot A_{x,x}.
\end{equation}\eqlabel{9.1.1}
For $M\in \Vv(X)_A^A$, the $(x,y)$-component of $\varepsilon^M:\ FG(M)\to M$ is
$$\varepsilon^M=\psi_{x,x,y}\circ (i\ot A_{x,y}):\ FG(M)_{x,y}=M^{{\rm co}A}_x\ot A_{x,y}\to M_{x,y}.$$
In order to show that $(F,G)$ is an adjoint pair, we have verify that
$$F(N)=\varepsilon^{F(N)}\circ F(\eta^N)~~{\rm and}~~G(M)=G(\varepsilon_M)\circ \eta^{G(M)},$$
for all $N\in \Dd(X)$ and $M\in \Vv(X)_A^A$. Now
\begin{eqnarray*}
&&\hspace*{-15mm}
\varepsilon^{F(N)}_{x,y}\circ F(\eta^N)_{x,y}= (N_x\ot m_{x.x.y})\circ (i\ot A_{x,y})\circ (\eta_x^N\ot A_{x,y})\\
&=& (N_x\ot m_{x,x,y})\circ (N_x\ot \eta_x\ot A_{x,y})= N_x\ot A_{x,y}=F(N)_{x,y},
\end{eqnarray*}
proving the first formula. For the second formula, we consider the diagram
$$\xymatrix{
M^{{\rm co}A}_x\ar[d]_{\eta^{G(M)}_x}\ar[rrd]^{M^{{\rm co}A}_x\ot \eta_x}&&\\
(M^{{\rm co}A}_x \ot A_{x,x})^{{\rm co} A_{x,x}}\ar[d]_{(i\ot A_{x,x})^{{\rm co} A_{x,x}}}\ar[rr]^{i}&&
M^{{\rm co}A}_x \ot A_{x,x}\ar[d]^{i\ot A_{x,x}}\\
(M_{x,x}\ot A_{x,x})^{{\rm co}A_{x,x}}\ar[d]_{\psi^{{\rm co}A_{x,x}}_{x,x,x}}\ar[rr]^{i}&&
M_{x,x}\ot A_{x,x}\ar[d]^{\psi_{x,x,x}}\\
M^{{\rm co}A}_x \ar[rr]^{i}&& M_x}$$
The commutativity of the triangle follows from the definition of $\eta^{G(M)}_x$; the commutativity of the
two squares follows from the definition of $G$ at the level of morphisms. Now
$$\psi_{x,x,x}\circ (i\ot A_{x,x})\circ (M^{{\rm co}A}_x\ot \eta_x)=
\psi_{x,x,x} \circ (M_{x,x}\ot \eta_x)\circ i=i,$$
and it follows from the uniqueness in the universal property of equalizers that the vertical composition
in the diagram is the identity on $M^{{\rm co}A}_x=G(M)_x$; this vertical composition is the
$x$-component of the right hand side in the second formula.
\end{proof}

Let $A$ be a $\Cog(\Vv)$-category, with underlying class $|A|=X$. For all
$x,y,z\in X$, we consider the canonical map
$$\can^z_{x,y}=(m_{z,x,y}\ot A_{x,})\circ (A_{z,x}\ot \Delta_{x,y}):\ A_{z,x}\ot A_{x,y}\to A_{z,y}\ot A_{x,y}.$$

With respect to the observations made at the end of \seref{6}, the following theorem should be compared to \cite[Theorem 7.14]{BL}.

\begin{theorem}\thlabel{9.2} {\bf (Fundamental Theorem for Hopf Modules)}
Let $\Vv$ be a strict braided monoidal category with equalizers.
For a $\Cog(\Vv)$-category $A$ with underlying class $X$, the following assertions
are equivalent.
\begin{enumerate}
\item $A$ is a Hopf $\Vv$-category;
\item the pair of adjoint functors $(F,G)$ from \prref{9.1} is a pair of inverse
equivalences between the categories $\Dd(X)$ and $\Vv(X)_A^A$;
\item the functor $G$ from \prref{9.1} is fully faithful;
\item $\can^z_{x,y}$ is an isomorphism, for all $x,y,z\in X$;
\item $\can^x_{x,y}$ has a left inverse $f_{x,y}$ and $\can^y_{x,y}$ is an isomorphism, with inverse $g_{x,y}$, for all $x,y\in X$.
\end{enumerate}
\end{theorem}

\begin{proof}
$\ul{(1)\Rightarrow (2)}$.
{\bf Part 1.} $\varepsilon^M$ has an inverse $\alpha^M$, for all $M\in \Vv(X)_A^A$.\\
We first show that the morphism
$$\gamma_{x,y}=\psi_{x,y,x}\circ (M_{x,y}\ot S_{x,y})\circ \rho_{x,y}:\ M_{x,y}\to M_{x,x}$$
satisfies the equality
\begin{equation}\eqlabel{9.2.1}
\rho_{x,x}\circ \gamma_{x,y}=(M_{x,x}\ot \eta_x)\circ \gamma_{x,x}.
\end{equation}
\begin{eqnarray*}
&&\hspace*{-15mm}
\rho_{x,x}\circ \gamma_{x,y}=\rho_{x,x}\circ \psi_{x,y,x}\circ (M_{x,y}\ot S_{x,y})\circ \rho_{x,y}\\
&\equal{\equref{9.0.1}}&
(\psi_{x,y,x}\ot m_{x,y,x})\circ (M_{x,y}\ot c_{A_{x,y},A_{y,x}}\ot A_{y,x})\circ (\rho_{x,y}\ot \Delta_{y,x})\\
&&\hspace*{5mm}\circ~ (M_{x,y}\ot S_{x,y})\circ \rho_{x,y}\\
&\equal{\equref{2.13.2}}&
(\psi_{x,y,x}\ot m_{x,y,x})\circ (M_{x,y}\ot c_{A_{x,y},A_{y,x}}\ot A_{y,x})\\
&&\hspace*{5mm}
\circ~ (M_{x,y}\ot A_{x,y}\ot c_{A_{y,x},A_{y,x}})
\circ (\rho_{x,y}\ot S_{x,y}\ot S_{x,y})\\
&&\hspace*{5mm}
\circ~ (M_{x,y}\ot\Delta_{x,y})
\circ \rho_{x,y}\\
&=&
(\psi_{x,y,x}\ot A_{x,x})\circ (M_{x,y}\ot A_{y,x}\ot m_{x,y,x})\circ (M_{x,y}\ot c_{A_{x,y}\ot A_{y,x},A_{y,x}})\\
&&\hspace*{5mm}
\circ~ (M_{x,y}\ot A_{x,y}\ot S_{x,y}\ot S_{x,y}) \circ \rho_{x,y}^3\\
&\equal{(*)}&
(\psi_{x,y,x}\ot A_{x,x})\circ (M_{x,y}\ot c_{A_{x,x},A_{y,x}})\circ (M_{x,y}\ot m_{x,y,x}\circ A_{y,x})\\
&&\hspace*{5mm}
\circ~ (M_{x,y}\ot A_{x,y}\ot S_{x,y}\ot S_{x,y})\circ (M_{x,y}\ot\Delta_{x,y}\ot A_{x,y}) \circ \rho_{x,y}^2\\
&\equal{\equref{2.2.1}}&
(\psi_{x,y,x}\ot A_{x,x})\circ (M_{x,y}\ot c_{A_{x,x},A_{y,x}})\circ (M_{x,y}\ot \eta_x\ot A_{y,x})\\
&&\hspace*{5mm} \circ~ (M_{x,y}\ot \varepsilon_{x,y}\ot S_{x,y}) \circ \rho_{x,y}^2\\
&=& (\psi_{x,y,x}\ot A_{x,x})\circ (M_{x,y}\ot A_{y,x}\ot \eta_x)\circ (M_{x,y}\ot S_{x,y})\circ \rho_{x,y}\\
&=& (M_{x,y}\ot \eta_x)\circ \psi_{x,y,x}\circ (M_{x,y}\ot S_{x,y})\circ \rho_{x,y}= (M_{x,x}\ot \eta_x)\circ \gamma_{x,x}.
\end{eqnarray*}
At $(*)$ we used the naturality of $c$ resulting in the commutative diagram
$$\xymatrix{
A_{x,y}\ot A_{y,x}\ot A_{y,x}\ar[d]_{m_{x,y,x}\ot A_{y,x}}\ar[rr]^{c_{A_{x,y}\ot A_{y,x},A_{y,x}}}&&
A_{y,x}\ot A_{x,y}\ot A_{y,x}\ar[d]^{A_{y,x}\ot m_{x,y,x}}\\
A_{x,x} \ot A_{y,x}\ar[rr]^{c_{A_{x,x},A_{y,x}}}&& A_{y,x} \ot A_{x,x}}$$
From \equref{9.2.1} and the universal property of equalizers, it follows that there is a unique morphism
$\tilde{\gamma}_{x,y}:\ M_{x,y}\to M^{{\rm co}A}_x$ such that $i\circ \tilde{\gamma}_{x,y}=\gamma_{x,y}$.\\
Now we are ready to define $\alpha^M:\ M\to FG(M)$. The $(x,y)$-component is
$$\alpha^M_{x,y}=(\tilde{\gamma}_{x,y}\ot A_{x,y})\circ \rho_{x,y}:\ M_{x,y}\to M^{{\rm co}A}_x\ot A_{x,y}.$$
\begin{eqnarray*}
\varepsilon^M_{x,y}\circ \alpha^M_{x,y}&=&
\psi_{x,x,y}\circ (i\ot A_{x,y})\circ (\tilde{\gamma}_{x,y}\ot A_{x,y})\circ \rho_{x,y}\\
&=& \psi^2_{x,y,x,y}\circ (M_{x,y}\ot S_{x,y}\ot A_{x,y})\circ \rho^2_{x,y}\\
&\equal{\equref{2.2.2}}&
\psi_{x,y,y}\circ (M_{x,y}\ot \eta_y)\circ (M_{x,y}\ot\varepsilon_{x,y})\circ \rho_{x,y}=M_{x,y}.
\end{eqnarray*}
The proof of the fact that $\alpha^M$ is also a left inverse of $\varepsilon^M$ is more involved. We first
compute
$$\rho_{x,y}\circ \psi_{x,x,y}\circ (i\ot A_{x,y}):\ M^{{\rm co}A}_x\ot A_{x,y}\to M_{x,y}\ot A_{x,y}.$$
\begin{eqnarray}
&&\hspace*{-15mm}
\rho_{x,y}\circ \psi_{x,x,y}\circ (i\ot A_{x,y})\nonumber\\
&\equal{\equref{9.0.1}}&
(\psi_{x,x,y}\ot M_{x,x,y})\circ (M_{x,x}\ot c_{A_{x,x},A_{x,y}}\ot A_{x,y})\nonumber\\
&&\hspace*{5mm}\circ~ (\rho_{x,x}\ot\Delta_{x,y})\circ (i\ot A_{x,y})\nonumber\\
&=&
(\psi_{x,x,y}\ot M_{x,x,y})\circ (M_{x,x}\ot c_{A_{x,x},A_{x,y}}\ot A_{x,y})\nonumber\\
&&\hspace*{5mm}\circ~ (M_{x,x}\ot\eta_x\ot \Delta_{x,y})\circ (i\ot A_{x,y})\nonumber\\
&=&
(\psi_{x,x,y}\ot M_{x,x,y})\circ (M_{x,x}\ot A_{x,y}\ot \eta_x\ot A_{x,y})\circ (i\ot \Delta_{x,y})\nonumber\\
&=&
(\psi_{x,x,y}\ot A_{x,y})\circ (i\ot \Delta_{x,y}).\eqlabel{9.2.2}
\end{eqnarray}
Our next step is to compute
\begin{eqnarray*}
&&\hspace*{-15mm}
i\circ \tilde{\gamma}_{x,y}\circ \psi_{x,x,y} \circ (i\ot A_{x,y})\\
&=& \psi_{x,y,x}\circ (M_{x,y}\ot S_{x,y})\circ \rho_{x,y}\circ \psi_{x,x,y} \circ (i\ot A_{x,y})\\
&\equal{\equref{9.2.2}}& 
\psi_{x,y,x}\circ (M_{x,y}\ot S_{x,y})\circ (\psi_{x,x,y}\ot A_{x,y})\circ (i\ot \Delta_{x,y})\\
&=&
\psi_{x,x,y,x}^2\circ (M_{x.x}\ot A_{x,y}\ot S_{x,y})\circ (M_{x.x}\ot \Delta_{x,y})\circ (i\ot A_{x,y})\\
&\equal{\equref{2.2.1}}&
\psi_{x,x,x}\circ (M_{x,x}\ot \eta_x)\circ (M_{x,x}\ot\varepsilon_{x,y})\circ (i\ot A_{x,y})\\
&=& i\ot \varepsilon_{x,y}=i\circ (M^{{\rm co}A}_x\ot \varepsilon_{x,y}).
\end{eqnarray*}
The universal property of equalizers tells us that there is a unique $f:\ M^{{\rm co}A}_x\ot A_{x,y}\to M^{{\rm co}A}_x$
such that $i\circ f=i\ot \varepsilon_{x,y}$. This implies that
\begin{equation}\eqlabel{9.2.3}
\tilde{\gamma}_{x,y}\circ \psi_{x,x,y} \circ (i\ot A_{x,y})=M^{{\rm co}A}_x\ot \varepsilon_{x,y}.
\end{equation}
Finally
\begin{eqnarray*}
&&\hspace*{-12mm}
\alpha^M_{x,y}\circ \varepsilon^M_{x,y}=
(\tilde{\gamma}_{x,y}\ot A_{x,y})\circ \rho_{x,y}\circ \psi_{x,x,y}\circ (i\ot A_{x,y})\\
&\equal{\equref{9.2.2}}&
(\tilde{\gamma}_{x,y}\ot A_{x,y})\circ (\psi_{x,x,y}\ot A_{x,y})\circ (i\ot \Delta_{x,y})\\
&=& (\tilde{\gamma}_{x,y}\ot A_{x,y})\circ (\psi_{x,x,y}\ot A_{x,y})\circ (i\ot A_{x,y}\ot A_{x,y}))\circ (M^{{\rm co}A}_x\ot \Delta_{x,y})\\
&\equal{\equref{9.2.3}}&(M^{{\rm co}A}_x\ot \varepsilon_{x,y}\ot A_{x,y}) \circ (M^{{\rm co}A}_x\ot \Delta_{x,y})= M^{{\rm co}A}_x.
\end{eqnarray*}

{\bf Part 2.} $\eta^N$ has an inverse $\beta^N$, for all $N\in \Dd(X)$.\\
The $x$-component of $\beta^N$ is
$$\beta^N_x=(N_x\ot\varepsilon_{x,x})\circ i:\ (N_x\ot A_{x,x})^{{\rm co}A_{x,x}}\to N_x$$
It is easy to see that
$$\beta^N_x\circ \eta^N_x=(N_x\ot\varepsilon_{x,x})\circ i\circ \eta^N_x
\equal{\equref{9.1.1}}(N_x\ot\varepsilon_{x,x})\circ (N_x\ot \eta_x)=N_x.$$
The universal property of the equalizer entails that there is only one endomorphism $f$ of $(N_x\ot A_{x,x})^{{\rm co}A_{x,x}}$
such that $i\circ f=i$, namely the identity. Now
\begin{eqnarray*}
&&\hspace*{-12mm}
i\circ \eta^N_x\circ \beta^N_x\equal{\equref{9.1.1}} (N_x\ot \eta_x) \circ (N_x\ot\varepsilon_{x,x})\circ i\\
&=&(N_x\ot\varepsilon_{x,x} \ot A_{x,x})\circ (N_x\ot A_{x,x}\ot \eta_x)\circ i\\
&=&(N_x\ot\varepsilon_{x,x} \ot A_{x,x})\circ (N_x\ot \Delta_{x,x})\circ i=i,
\end{eqnarray*}
so it follows that $\eta^N_x\circ \beta^N_x=(N_x\ot A_{x,x})^{{\rm co}A_{x,x}}$.\\

$\ul{(2)\Rightarrow (3)}$ is obvious.\\

$\ul{(3)\Rightarrow (4)}$. For every $z\in X$, consider the object $M^z\in \Vv(X)$ given by
$M^z_{x,y}=A_{z,y}\ot A_{x,y}$. The structure morphisms $\rho^z_{x,y}=A_{z,y}\ot \Delta_{x,y}:\ M^z_{x,y}
\ot A_{x,y}$ and
\begin{eqnarray*}
&&\hspace*{-15mm}
\psi^z_{x,y,u}=(m_{z,y,u}\ot m_{x,y,u})\circ (A_{z,y}\ot c_{A_{x,y},A_{y,u}}\ot A_{y,u})\circ (M^z_{x,y}\ot \Delta_{y,u})\\
&:& M^z_{x,y}\ot A_{y,u}\to M^z_{x,u}
\end{eqnarray*}
make $M^z$ into an object of $\Vv(X)_A^A$. Let us verify that the compatibility relation \equref{9.0.1} holds. We
compute both sides of the equation, and see that they are equal.
\begin{eqnarray*}
&&\hspace*{-10mm}
\rho^z_{x,u}\circ \psi^z_{x,y,u}
=
(A_{z,u}\ot \Delta_{x,u})\circ (m_{z,y,u}\ot m_{x,y,u})
\circ (A_{z,y}\ot c_{A_{x,y},A_{y,u}}\ot A_{y,u})\\
&&\hspace*{5mm}\circ~ (A_{z,y}\ot A_{x,y}\ot \Delta_{y,u})\\
&=&
(m_{z,y,u}\ot m_{x,y,u}\ot m_{x,y,u})\circ (A_{z,y}\ot A_{y,u}\ot A_{x,y}\ot c_{A_{x,y},A_{y,u}}\ot A_{y,u})\\
&&\hspace*{5mm}
\circ~ (A_{z,y}\ot A_{y,u}\ot \Delta_{x,y}\ot \Delta_{y,u})\circ (A_{z,y}\ot c_{A_{x,y},A_{y,u}}\ot A_{y,u})\\
&&\hspace*{5mm}\circ~ (A_{z,y}\ot A_{x,y}\ot \Delta_{y,u})\\
&=&
(m_{z,y,u}\ot m_{x,y,u}\ot m_{x,y,u})\circ (A_{z,y}\ot A_{y,u}\ot A_{x,y}\ot c_{A_{x,y},A_{y,u}}\ot A_{y,u})\\
&&\hspace*{5mm}
\circ~ (A_{z,y}\ot c_{A_{x,y}\ot A_{x,y},A_{y,u}}\ot A_{y,u}\ot A_{y,u})\circ (A_{z,y}\ot \Delta_{x,y}\ot A_{y,u}\ot \Delta_{y,u})\\
&&\hspace*{5mm}\circ~ (A_{z,y}\ot A_{x,y}\ot \Delta_{y,u})\\
&=&
(m_{z,y,u}\ot m_{x,y,u}\ot m_{x,y,u})\circ (A_{z,y}\ot c_{A_{x,y},A_{y,u}}\ot c_{A_{x,y},A_{y,u}}\ot A_{y,u})\\
&&\hspace*{5mm}
\circ~ (A_{z,y}\ot A_{x,y}\ot c_{A_{x,y},A_{y,u}}\ot A_{y,u}\ot A_{y,u})\circ (A_{z,y}\ot \Delta_{x,y}\ot \Delta_{y,u}^2);\\
&&\hspace*{-10mm}
(\psi^z_{x,y,u}\ot m_{x,y,u})\circ (A_{z,y}\ot A_{x ,y}\ot c_{A_{x,y},A_{y,u}}\ot A_{y,u})\circ(\rho^z_{x,y}\ot\Delta_{y,u})\\
&=&
(m_{z,y,u}\ot m_{x,y,u}\ot m_{x,y,u})\circ (A_{z,y}\ot c_{A_{x,y},A_{y,u}}\ot A_{y,u}\ot A_{x,y}\ot A_{y,u})\\
&&\hspace*{5mm}
\circ~ (A_{z,y}\ot A_{x,y}\ot \Delta_{y,u}\ot A_{x,y}\ot A_{y,u})\\
&&\hspace*{5mm}
\circ~ (A_{z,y}\ot A_{x,y}\ot c_{A_{x,y},A_{y,u}}\ot A_{y,u})
\circ (A_{z,y}\ot \Delta_{x,y}\ot\Delta_{y,u})\\
&=&
(m_{z,y,u}\ot m_{x,y,u}\ot m_{x,y,u})\circ (A_{z,y}\ot c_{A_{x,y},A_{y,u}}\ot A_{y,u}\ot A_{x,y}\ot A_{y,u})\\
&&\hspace*{5mm}
\circ~ (A_{z,y}\ot A_{x,y}\ot c_{A_{x,y},A_{y,u}\ot A_{y,u}}\ot A_{y,u})
\\
&&\hspace*{5mm}
\circ~ (A_{z,y}\ot A_{x,y}\ot A_{x,y}\ot \Delta_{y,u}\ot A_{y,u})
\circ (A_{z,y}\ot \Delta_{x,y}\ot\Delta_{y,u})\\
&=&
(m_{z,y,u}\ot m_{x,y,u}\ot m_{x,y,u})\circ (A_{z,y}\ot c_{A_{x,y},A_{y,u}}\ot c_{A_{x,y},A_{y,u}}\ot A_{y,u})\\
&&\hspace*{5mm}
\circ~ (A_{z,y}\ot A_{x,y}\ot c_{A_{x,y},A_{y,u}}\ot A_{y,u}\ot A_{y,u})\circ (A_{z,y}\ot \Delta_{x,y}\ot \Delta_{y,u}^2).
\end{eqnarray*}
Consider the morphism $f=A_{z,x}\ot\eta_{x,x}:\ A_{z,x}\to A_{z,x}\ot A_{x,x}=M^z_{x,x}$. Since
\begin{eqnarray*}
\rho^z_{x,x}\circ f&=& (A_{z,x}\ot \Delta_{x,x})\circ (A_{z,x}\ot\eta_{x,x})=(A_{z,x}\ot\eta_{x,x}\ot \eta_x)\\
&=& (A_{z,x}\ot A_{x,x}\ot \eta_x)\circ (A_{z,x}\ot \eta_x)=(M^z_{x,x}\ot \eta_x)\circ f,
\end{eqnarray*}
there exists a unique $\tilde{f}:\ A_{z,x}\to M^{z{\rm co}A}_x$ such that $i\circ \tilde{f}=f$. $\tilde{f}$ is invertible,
with inverse $g=(A_{z,x}\ot \varepsilon_{x,x})\circ i$. Indeed,
$$g\circ \tilde{f}= (A_{z,x}\ot \varepsilon_{x,x})\circ f= (A_{z,x}\ot \varepsilon_{x,x})\circ (A_{z,x}\ot\eta_{x,x})=A_{z,x}.$$
We also have that
\begin{eqnarray*}
i\circ \tilde{f}\circ g&=& f\circ g= (A_{z,x}\ot\eta_{x,x})\circ (A_{z,x}\ot \varepsilon_{x,x})\circ i\\
&=& (A_{z,x}\ot\varepsilon_{x,x}\ot A_{x,x})\circ (A_{z,x}\ot A_{x,x}\ot \eta_x)\circ i\\
&=& (A_{z,x}\ot\varepsilon_{x,x}\ot A_{x,x})\circ (A_{z,x}\ot \Delta_{x,x})\circ i=i,
\end{eqnarray*}
and it follows from the uniqueness in the universal property of equalizers that $\tilde{f}\circ g=M^{z{\rm co}A}_x$.
We know by assumption that
$$\varepsilon^z_{x,y}=\psi^z_{x,x,y}\circ (i\ot A_{x,y}):\ M^{z{\rm co}A}_x\to  A_{x,y}\to M^z_{x,y}$$
is an isomorphism. It follows that
\begin{eqnarray*}
&&\hspace*{-15mm}
\varepsilon^z_{x,y}\circ (\tilde{f}\ot A_{x,y})
= (m_{z,x,y}\ot m_{x,x,y})\circ (A_{z,x}\ot c_{A_{x,x},A_{x,y}}\ot A_{x,y})\\
&&\hspace*{5mm}\circ~ (A_{z,x}\ot A_{x,x}\ot \Delta_{x,y}) \circ (i\ot A_{x,y})\circ (\tilde{f}\ot A_{x,y})\\
&=&
(m_{z,x,y}\ot m_{x,x,y})\circ (A_{z,x}\ot c_{A_{x,x},A_{x,y}}\ot A_{x,y})\\
&&\hspace*{5mm}\circ ~(A_{z,x}\ot A_{x,x}\ot \Delta_{x,y}) \circ (A_{x,x}\ot \eta_{x,x}\ot A_{x,y})\\
&=&(m_{z,x,y}\ot m_{x,x,y})\circ (A_{z,x}\ot c_{A_{x,x},A_{x,y}}\ot A_{x,y})\\
&&\hspace*{5mm}\circ~ (A_{x,x}\ot \eta_{x,x}\ot A_{x,y}\ot A_{x,y})\circ (A_{z,x}\ot \Delta_{x,y})\\
&=&(m_{z,x,y}\ot m_{x,x,y})\circ (A_{z,x}\ot A_{x,y}\ot \eta_x\ot A_{x,y})\circ (A_{z,x}\ot \Delta_{x,y})\\
&=& (m_{z,x,y}\ot A_{x,y})\circ (A_{z,x}\ot \Delta_{x,y})=\can^z_{x,y}
\end{eqnarray*}
is an isomorphism.\\

$\ul{(4)\Rightarrow (5)}$ is obvious.\\

$\ul{(5)\Rightarrow (1)}$.\\
We define the antipode as follows:
$$S_{x,y}=(A_{y,x}\ot\varepsilon_{x,y})\circ g_{x,y}\circ (\eta_y\ot A_{x,y}).$$
We have to show that the equations (\ref{eq:2.2.1}-\ref{eq:2.2.2}) are satisfied. To this end, we first need
some auxiliary formulas.
Composing the equality
\begin{eqnarray*}
&&\hspace*{-15mm}
(m_{x,y,y}\ot A_{x,y})\circ (A_{x,y}\ot \can^y_{x,y})\\
&=& (m_{x,y,y}\ot A_{x,y})\circ (A_{x,y}\ot m_{y,x,y}\ot A_{x,y})\circ (A_{x,y}\ot A_{y,x}\circ \Delta_{x,y})\\
&=& (m_{x,x,y}\ot A_{x,y})\circ (m_{x,y,x}\ot A_{x,y}\ot A_{x,y}) \circ (A_{x,y}\ot A_{y,x}\circ \Delta_{x,y})\\
&=& (m_{x,x,y}\ot A_{x,y})\circ (A_{x,x}\ot \Delta_{x,y})\circ (m_{x,y,x}\ot A_{x,y})\\
&=& \can^x_{x,y} \circ (m_{x,y,x}\ot A_{x,y})
\end{eqnarray*}
to the left with $f_{x,y}$ and to the right with $A_{x,y}\ot g_{x,y}$, we find that
\begin{equation}\eqlabel{9.2.4}
f_{x,y}\circ (m_{x,y,y}\ot A_{x,y})=(m_{x,y,x}\ot A_{x,y})\circ (A_{x,y}\ot g_{x,y}).
\end{equation}
Composing the equality
\begin{eqnarray*}
&&\hspace*{-15mm}
(\can^y_{x,y}\ot A_{x,y})\circ (A_{y,x}\ot \Delta_{x,y})\\
&=& 
(m_{y,x,y}\ot A_{x,y}\ot A_{x,y})\circ (A_{y,x}\ot \Delta_{x,y}\ot A_{x,y})\circ (A_{y,x}\ot \Delta_{x,y})\\
&=& 
(m_{y,x,y}\ot A_{x,y}\ot A_{x,y})\circ (A_{y,x}\ot A_{x,y}\ot \Delta_{x,y}) \circ (A_{y,x}\ot \Delta_{x,y})\\
&=&
(A_{y,y} \ot \Delta_{x,y}) \circ (m_{y,x,y}\ot A_{x,y}) \circ (A_{y,x}\ot \Delta_{x,y})\\
&=& 
(A_{y,y} \ot \Delta_{x,y}) \circ \can^y_{x,y}
\end{eqnarray*}
to the left and to the right with $g_{x,y}$, we find that
\begin{equation}\eqlabel{9.2.5}
(A_{y,x}\ot \Delta_{x,y}) \circ g_{x,y}= (g_{x,y}\ot A_{x,y})\circ (A_{y,y} \ot \Delta_{x,y}).
\end{equation}
\begin{eqnarray*}
&&\hspace*{-13mm}
\eta_x\circ\varepsilon_{x,y}=(A_{x,x}\ot \varepsilon_{x,y})\circ (\eta_x\ot A_{x,y})\\
&=& (A_{xx,}\ot \varepsilon_{x,y})\circ f_{x,y}\circ \can^x_{x,y} \circ (\eta_x\ot A_{x,y})\\
&=& (A_{xx,}\ot \varepsilon_{x,y})\circ f_{x,y}\circ (m_{x,x,y}\ot A_{x,y})\circ (A_{x,x}\ot \Delta_{x,y})  \circ (\eta_x\ot A_{x,y})\\
&=& (A_{xx,}\ot \varepsilon_{x,y})\circ f_{x,y}\circ (m_{x,x,y}\ot A_{x,y})\circ (\eta_x\ot A_{x,y}\ot A_{x,y})\circ \Delta_{x,y}\\
&=& (A_{xx,}\ot \varepsilon_{x,y})\circ f_{x,y}\circ (m_{x,x,y}\ot A_{x,y})\circ (A_{x,y}\ot \eta_y\ot A_{x,y})\circ \Delta_{x,y}\\
&\equal{\equref{9.2.4}}&
(A_{xx,}\ot \varepsilon_{x,y})\circ (m_{x,y,x}\ot A_{x,y})\circ (A_{x,y}\ot g_{x,y})\\
&&\hspace*{5mm}\circ~ (A_{x,y}\ot \eta_y\ot A_{x,y})\circ \Delta_{x,y}\\
&=& m_{x,y,x}\circ (A_{x,y}\ot S_{x,y})\circ \Delta_{x,y},
\end{eqnarray*}
and this shows that \equref{2.2.1} holds.
\begin{eqnarray*}
&&\hspace*{-13mm}
\eta_y\circ\varepsilon_{x,y}=(A_{y,y}\ot \varepsilon_{x,y})\circ (\eta_y\ot A_{x,y})\\
&=&
(A_{y,y}\ot \varepsilon_{x,y})\circ \can^y_{x,y}\circ g_{x,y}\circ (\eta_y\ot A_{x,y})\\
&=&
(A_{y,y}\ot \varepsilon_{x,y})\circ (m_{y,x,y}\ot A_{x,y})\circ (A_{y,x}\ot \Delta_{x,y})\circ g_{x,y}\circ (\eta_y\ot A_{x,y})\\
&=&
m_{y,x,y}\circ (A_{y,x}\ot A_{x,y}\ot \varepsilon_{x,y}) \circ (A_{y,x}\ot \Delta_{x,y})\circ g_{x,y}\circ (\eta_y\ot A_{x,y})\\
&=&
m_{y,x,y}\circ (A_{y,x}\ot \varepsilon_{x,y}\ot A_{x,y}) \circ (A_{y,x}\ot \Delta_{x,y})\circ g_{x,y}\circ (\eta_y\ot A_{x,y})\\
&\equal{\equref{9.2.5}}&
m_{y,x,y}\circ (A_{y,x}\ot \varepsilon_{x,y}\ot A_{x,y}) \circ (g_{x,y}\ot A_{x,y})\\
&&\hspace*{5mm}\circ (A_{y,y} \ot \Delta_{x,y})\circ (\eta_y\ot A_{x,y})\\
&=&
m_{y,x,y}\circ (A_{y,x}\ot \varepsilon_{x,y}\ot A_{x,y}) \circ (g_{x,y}\ot A_{x,y})\circ (\eta_y\ot A_{x,y}\ot A_{x,y})\circ \Delta_y\\
&=& m_{y,x,y}\circ~ (S_{x,y} \ot A_{x,y}) \circ \Delta_y,
\end{eqnarray*}
and this shows that \equref{2.2.2} holds.
\end{proof}

\begin{remarks}
1) The implication $\ul{(1)\Rightarrow (4)}$ can easily be proved directly: it is easily verified that
$$(\can^z_{x,y})^{-1}=(m_{z,y,x}\circ A_{x,y})\circ (A_{z,y}\ot S_{x,y}\ot A_{x,y})\circ (A_{z,y}\ot \Delta_{x,y}).$$
2) It follows from the Theorem that a Hopf module over a Hopf category is isomorphic to a free Hopf module,
that is a Hopf module in the image of the functor $G$. This result is known in the literature as the
Fundamental Theorem for Hopf modules. Its original form (in the case where $\Vv$ is de category of vector
spaces and $X$ is a singleton) it is due to Larson and Sweedler \cite{LS}, see also \cite[Theorem 1.1]{Sweedler}.
For the case where $\Vv$ is an arbitrary braided monoidal category with equalizers and $X$ is a singleton,
see \cite[Theorem 3.4]{Takeuchi} and \cite[Theorem 1.4]{Lyubashenko}.
\end{remarks}

Let us now proceed to some applications of the Fundamental Theorem. We restrict attention to the case where
$\Vv$ is the category $\Mm_k^{\rm f}$ of finitely generated projective modules over a commutative ring $k$
(or finite dimensional vector spaces over a field $k$). Our applications generalize applications of the classical
Fundamental Theorem as they can be found in \cite[Chapter 4]{Sweedler}.\\
For $\Vv=\Mm_k^{\rm f}$, the axioms (\ref{eq:2.2.1}-\ref{eq:2.2.2})
take the following form
\begin{equation}\eqlabel{2.3.1}
h_{(1)}S_{x,y}(h_{(2)})=\varepsilon_{x,y}(h)1_x~~;~~S_{x,y}(h_{(1)})h_{(2)}=\varepsilon_{x,y}(h)1_y,
\end{equation}
for all $x,y\in X$ and $h\in A_{x,y}$. The formula (\ref{eq:2.13.1}-\ref{eq:2.13.2}) can be written as
\begin{eqnarray}
&&S_{x,z}(hl)=S_{y,z}(l)S_{x,y}(h);\eqlabel{2.3.2}\\
&&\Delta_{y,x}(S_{x,y}(h))=S_{x,y}(h_{(2)})\ot S_{x,y}(h_{(1)}),\eqlabel{2.3.3}
\end{eqnarray}
for all $x,y,z\in X$, $h\in A_{x,y}$ and $l\in A_{y,z}$. The compatibility relation for Hopf modules
amounts to
\begin{equation}\eqlabel{7.1.1}
\rho_{x,z}(ma)=m_{[0]}a_{(1)}\ot m_{[1]}a_{(2)},
\end{equation}
for all $m\in M_{x,y}$ and $a\in A_{y,z}$. 

\begin{proposition}\prlabel{7.3}
Let $A$ be a Hopf category in $\Mm_k^{\rm f}(X)$. Then $A^*$ is a Hopf module,
with structure maps $\rho_{x,y}:\ A^*_{x,y}\to A^*_{x,y}\ot A_{x,y}$
and $\psi_{x,y,z}:\ A^*_{x,y}\ot A_{y,z}\to A^*_{x,z}$ defined as follows:
\begin{enumerate}
\item For $a^*\in A^*_{x,y}$,
$\rho_{x,y}(a^*)=\sum_i a^*a^*_i\ot a_i$,
where $\sum_i a^*_i\ot a_i\in A^*_{x,y}\ot A_{x,y}$ is the finite dual basis of $A_{x,y}$.
The multiplication on $A^*_{x,y}$ is the opposite convolution.
\item For $a^*\in A^*_{x,y}$ and $a\in A_{y,z}$, $\psi_{x,y,z}(a^*\ot a)=
a^*\rightact a\in A^*_{x,z}$ is given by the formula
$\lan a^*\rightact a, b\ran=\lan a^*,bS_{y,z}(a)\ran$, for all $b\in A_{y,z}$.
\end{enumerate}
\end{proposition}

\begin{proof}
The right $A$-coaction is obtained as follows: $A_{x,y}$ is a $k$-coalgebra,
hence $A_{x,y}^*$ is a $k$-algebra (with opposite convolution product).
It is therefore a right $A_{x,y}^*$-module, and a right $A_{x,y}$-comodule.
The coaction that is opbtained in this way is precisely the one that is described
in the Proposition.\\
Now let us show that the structure maps $\psi_{x,y,z}$ define a right $A$-module
structure on $A^*$.\\
\ul{Associativity}. For all $a^*\in A^*_{x,y}$, $a\in A_{y,z}$, $b\in A_{z,u}$ and
$c\in A_{x,u}$, we have that
\begin{eqnarray*}
\lan a^*\rightact (ab),c\ran&=&\lan a^*, c S_{y,u}(ab)\ran\equal{\equref{2.3.2}}
\lan a^*, cS_{z,u}(b)S_{y,z}(a)\ran\\
&=&\lan a^*\rightact a, cS_{z,u}(b)\ran=\lan (a^*\rightact a)\rightact b,c\ran.
\end{eqnarray*}
\ul{Unit property}.
For all $a^*\in A^*_{x,y}$ and $a\in A_{x,y}$, we have that
$$\lan a^*\rightact 1_y,a\ran=\lan a^*,aS_{y,y}(1_i)\ran \equal{\equref{2.3.2}}
\lan a^*,a\ran.$$
Now we verify the Hopf compatibility condition \equref{7.1.1}. We have to show
that
$$\rho_{x,z}(a^*\rightact a)=\sum_i (a^*a^*_i)\rightact a_{(1)}\ot a_i a_{(2)},$$
for all $a^*\in A^*_{x,y}$ and $a\in A_{y,z}$. Now
$$\rho_{x,z}(a^*\rightact a)=\sum_j (a^*\rightact a)b^*_j \ot b_j,$$
where $\sum_j b^*_j \ot b_j\in A^*_{x,z}\ot A_{x,z}$ is the dual basis of $A_{x,z}$,
so it suffices to show that
$$\sum_j \lan (a^*\rightact a)b^*_j, c\ran b_j=\sum_i \lan (a^*a^*_i)\rightact a_{(1)},c\ran a_i a_{(2)},$$
for all $c\in A_{x,z}$. This can be done as follows:
\begin{eqnarray*}
&&\hspace*{-2cm}
\sum_i \lan (a^*a^*_i)\rightact a_{(1)},c\ran a_i a_{(2)}
\equal{\equref{2.3.3}}
\sum_i \lan a^*,c_{(2)}S_{y,z}(a_{(1)})\ran \lan a_i^*, c_{(1)}S_{y,z}(a_{(2)})\ran a_ia_{(3)}\\
&=& \lan a^*,c_{(2)}S_{y,z}(a_{(1)})\ran c_{(1)}S_{y,z}(a_{(2)})a_{(3)}\\
&\equal{\equref{2.3.1}}&
\lan a^*,c_{(2)}S_{y,z}(a_{(1)})\ran c_{(1)} \varepsilon_{y,z}(a_{(2)})1_z
= \lan a^*,c_{(2)}S_{y,z}(a)\ran c_{(1)}\\
&=& \sum_j \lan a^*\rightact a, c_{(2)}\ran \lan b_j^*,c_{(1)}\ran b_j
= \sum_j \lan (a^*\rightact a)b^*_j, c\ran b_j.
\end{eqnarray*}
\end{proof}

We compute $A^{*{\rm co}A}$. Recall that $A_{x,x}$ is a Hopf algebra, for every $x\in X$,
and that
$$A^{*{\rm co}A}_x=(A^*_{x,x})^{{\rm co}A_{x,x}}=\int^l_{A^*_{x,x}}=
\{\varphi\in A^*_{x,x}~|~\varphi a^*=\lan a^*,1_x\ran \varphi,~{\rm for~all}~a^*\in A^*_{x,x}\},$$
the space of left integrals on $A_{x,x}$. From \thref{9.2} and \prref{7.3}, we obtain the
following result.

\begin{corollary}\colabel{7.4}
Let $A$ be a Hopf category in $\Mm_k^{\rm f}(X)$. For all $x,y\in X$, we have an isomorphism
$$\alpha_{x,y}=\varepsilon_{x,y}^{A^*}:\ \int^l_{A^*_{x,x}}\ot A_{x,y}\to A_{x,y}^*,~~
\varepsilon_{x,y}^{A^*}(\varphi\ot a)=\varphi\rightact a.$$
\end{corollary}

\begin{proposition}\prlabel{7.5}
Let $A$ be a Hopf category in $\Mm_k^{\rm f}(X)$. The antipode maps $S_{x,y}:\ A_{x,y}\to A_{y,x}$
are bijective, for all $x,y\in X$.
\end{proposition}

\begin{proof}
It is well-known (and it also follows from \coref{7.4}) that $J=\int^l_{A^*_{x,x}}$
is finitely generated projective of rank one as a $k$-module. Therefore the evaluation map
$$\ev:\ J^*\ot J\to k,~~\ev(p\ot \varphi)=p(\varphi)$$
is an isomorphism of $k$-modules. The isomorphism
$$\tilde{\alpha}_{x,y}=(J^*\ot \alpha)\circ (\ev^{-1}\ot A_{x,y}):\ A_{x,y}\to J^*\ot A_{x,y}$$
can be described explicitly as follows:
$$\tilde{\alpha}_{x,y}(a)=\sum_l p_l\ot \varphi_l\rightact a,$$
where $\ev^{-1}(1)=\sum_l p_l\ot \varphi_l$.\\
Now assume that $S_{x,y}(a)=0$, for some $a\in A_{x,y}$. For all $\varphi\in A^*_{x,x}$ and $b\in A_{x,y}$, we have
that
$$\lan \varphi\rightact a,b\ran=\lan \varphi,bS_{x,y}(a)\ran=0,$$
so it follows that $\tilde{\alpha}_{x,y}(a)=0$, and $a=0$, since $\tilde{\alpha}_{x,y}$ is
a bijection. This proves that $S_{x,y}$ is injective.\\
Now assume that $k$ is a field. The maps
$$\alpha=S_{x,y}\circ S_{y,x}~~{\rm and}~~\beta=S_{y,x}\circ S_{x,y}$$
are injective endomorphisms of the finite dimensional vector spaces $A_{y,x}$ and
$A_{x,y}$. From the dimension formulas, it follows that they are automorphisms. We then have
that
\begin{eqnarray*}
A_{y,x}&=& \alpha\circ \alpha^{-1}=S_{x,y}\circ S_{y,x}\circ \alpha^{-1};\\
A_{y,x}&=& \beta^{-1}\circ \beta= \beta^{-1}\circ S_{y,x}\circ S_{x,y}.
\end{eqnarray*}
This tells us that $S_{x,y}$ has a left inverse and a right inverse; these are necessarily equal,
hence $S_{x,y}$ is bijective.\\
Now consider the general case where $k$ is a commutative ring. The surjectivity of $S_{x,y}$ follows
from a local-global argument. Let $Q=\Coker(S_{x,y})$. For every prime ideal $p$ of $k$, we can
consider the localized Hopf category $A_p$, with $A_{p,x,y}=A_{x,y}\ot k_p$.
For every prime ideal $p$ of $k$,
$\Coker(S_{p,x,y})=Q_p$, since localization at a prime ideal is an
exact functor. Now the spaces $A_{p,x,y}/pA_{p,x,y}$ define a finite dimensional Hopf
category $A_p/pA_p$ over the field $k_p/pk_p$, and its antipode maps are bijective.
It follows from Nakayama's Lemma that the localized maps $S_{p,x,y}:\ A_{p,x,y}\to A_{p,y,x}$
are all bijective, and then it follows that $S_{x,y}$ is bijective.
\end{proof}

\end{document}